\documentclass[12pt,reqno]{amsart}

\usepackage{epsfig}
\usepackage{varioref}
\usepackage{amsmath}
\usepackage{amsfonts}
\usepackage{amsthm}

\usepackage{amssymb}

\def\sqr#1#2{{\vcenter{\vbox{\hrule height .#2pt
     \hbox{\vrule width .#2pt height#1pt \kern#1pt \vrule
     width .#2pt} \hrule height .#2pt}}}}

\newcommand{\R}{{\mathbb{R}}}

\newcommand{\C}{{\mathbb{C}}}
\newcommand{\N}{{\mathbb{N}}}

\newcommand{\Ep}{{\mathcal{E}}}
\newcommand{\Rea}{\operatorname{Re}}

\newcommand{\cG}{{\mathcal{G}}}

\newcommand{\Ba}{{\mathcal{B}}}
\newcommand{\Hi}{{\mathcal{H}}}

\newcommand{\Fo}{{\mathcal{F}}}
\newcommand{\Fin}{{\mathcal{F}}(I)}
\newcommand{\Cil}{{\mathcal{C}}}

\newcommand{\Mi}{{\mathcal{M}}}
\newcommand{\di}{{\mathcal{D}}}

\newtheorem{theorem}{Theorem}
\newtheorem{Remark}{Remark}
\newtheorem{definition}{Definition}
\newtheorem{example}{Example}
\newtheorem{lemma}{Lemma}
\newtheorem{proposition}{Proposition}

\begin{document}


\title{A unified approach to infinite dimensional integration}

\author{S. Albeverio $^*$ and S. Mazzucchi $^{**}$ }

\address{ $^*$ Institut f\"ur Angewandte Mathematik,
Endenicher Allee 60, 53115 Bonn, HCM; BIBOS; IZKS; Cerfim (Locarno). }
\email{albeverio@iam.uni-bonn.de}

\address{$^{**}$ Dipartimento di Matematica, Universit\`a  di Trento, via Sommarive 14
I-38123 Trento, Italy.}
\email{ sonia.mazzucchi@unitn.it}


\begin{abstract}
An approach to infinite dimensional integration which unifies the case of oscillatory integrals and the case of probabilistic type integrals is presented. It provides a truly infinite dimensional construction of integrals as linear functionals, as much as possible independent of the underlying topological and measure theoretical structure. Various applications are given, including, next to Schr\"odinger and diffusion equations, also higher order hyperbolic and parabolic equations.\\


\end{abstract}


\maketitle


\section{Introduction}

Integration has its origins in the problem of determination of volumes (besides the one of finding primitives). Since Cauchy's first systematic investigation, where continuous functions on bounded intervals are considered for integration, an increasing generality has been pursued, both concerning admissible functions and domains of integration, culminating first in the concept of Riemann integrable functions and the relative integral, the underlying measure still being essentially the ``natural'' volume element, see, e.g. , \cite{Pes,Boy}. \\
The increasing generality reached, especially after Cantor's work, in the analysis of the structure of sets and functions led to an approach, particularly associated to Lebesgue, where the underlying measure, rather than the functions to be integrated, are the basic quantities to be considered. This finally has brought about what is known nowadays as ``abstract integration theory'' based on measure spaces, which since Kolmogorov \cite{Ko,Kol} is also at the basis of probability theory.\\
 Despite the great generality of this approach, the integrals arising in certain area of mathematics and its applications do not quite fit into this framework. In particular Lebesgue integrals (and classical integrals on measure spaces) are by construction absolutely convergent integrals, whereas there are areas and problems where one would not necessarily have this strong property, see, e.g., \cite{Pes,Mul1,Muld2}. Of particular interest for the present paper are integrals of oscillatory functions, where compensations between contributions of different signs can lead to finite overall results for the integrals, whereas the corresponding absolute integrals diverge. Such  non absolutely convergent integrals are usually conceived as ``improper integrals'' and handled as limits of ``proper integrals'', whenever the limits exist.\\
These improper integrals were traditionally not studied systematically, indeed they were quite marginal within theories of integration. Yet oscillatory functions and integrals of them are of great importance in applications, especially in connection with waves (optics, electromagnetism, hydrodynamics) and oscillatory (vibrations) phenomena, as they occur in mechanical or quantum mechanical structures.
Already in the XIX-century the behavior of the oscillatory integrals of interest was analyzed both numerically and analytically in certain asymptotic regimes, which, e.g., in optics happen to be critical points or manifolds. These stationary (or saddle point) methods to handle (finite dimensional) oscillatory integrals were first developed according to problems at hand, but in the second half of last century they became special cases of a general theory under the name of theory of Fourier integral operators, see, e.g. \cite{Hor1,Hor2,DuiH,GuSt,HeRo,HewStr,Ste,Maslov}. 
The detailed ``method of stationary phase'' (in the case where the critical manifolds are degenerate) has extremely interesting connections with the theory of resolution of singularities of mappings (and catastrophe theory), see, e.g., \cite{Arnoldetal,Co,DuiH,AlBr1,AlHK,AMBull}.\\
Due to the fact that oscillatory integrals are finite but in general not absolutely convergent, looked upon as linear functionals they have their own definition domains and continuity properties, different from those of the absolutely convergent integrals with respect to measures on measurable spaces.
This is so on finite dimensional vector spaces (or, more generally, on locally compact spaces, where Lebesgue's integration theory was further developed). So even in finite dimensional integration theory we should distinguish between ``oscillatory integrals'' (typically not absolutely convergent) and ``absolutely convergent integrals''. The latter are handled by abstract Lebesgue's integration theory, where the underlying measure structure is relevant, but they can also be looked upon as continuous linear functionals (an approach stressed particularly by P.J. Daniell, see e.g. \cite{Pes}).
Before leaving the ``finite dimensional theory'' let us also mention that Henstock-Kurzweil-Denjoy-Perron integration theory, see, e.g. \cite{Mul1,Muld2,Sc}, extends Lebesgue's integration theory in a certain direction covering also non absolutely convergent integrals. A systematic theory of finite dimensional oscillatory integrals in this setting has however not been developed yet.\\
Already at the beginning of the XXth century, especially after the work of Borel, Fr\'echet and Lebesgue, interest in extending the integration theory (and in fact analysis as a whole) to infinite dimensional spaces arose. 
The major momentum in developing infinite dimensional integration and analysis came in connection with the study of strongly irregular natural phenomena.
Mathematicians like V. Volterra, N. Wiener and N. Kolmogorov were mainly motivated by the study of such phenomena and the search for appropriate mathematical methods to understand and describe them. In particular N. Wiener \cite{Wie} introduced his probability measure on paths to provide a mathematical setting for the description of the behavior of the physical Brownian motion (partly anticipated in this by Einstein, Smoluchowski in physics, and independently, on the other hand, Bachelier, who was rather concerned with problems in economics). V. Volterra stressed the importance of infinite dimensional analysis in particular in connection with certain variational principles. Kolmogorov introduced and developed measures on product spaces, by means of projective systems (products of infinitely many probability measures had already been considered before by Daniell and Steinhaus, see also e.g. \cite{ChaRam,Lamb}).
Since Wiener's work and especially after Cameron, Martin and Kakutani's discovery of the close connection of Wiener measure with potential theory and parabolic equations,  probability measures on Hilbert and Banach spaces began to be studied systematically, see, e. g., references in  \cite{Bog,Bog2,Vak,Linde}. In particular Wiener measure was recognized as being just a special case of Gaussian measures on Banach spaces, with underlying determining ``tangential'' Hilbert spaces. This brought about ``in the west'' the theory of abstract Wiener Spaces \cite{Gro1,Gro2,DuFeLeC,HY,Kal,Kuo,Ba-Che,Doo,Doo2,Partha,Roz} and ``in the east'' the theory of measures on the dual of nuclear spaces \cite{Mi,GelVi,Sko,Sa,VeSu}.
Further impetus came on one hand through the development of the theory of stochastic differential equations (where K. It\^{o} and Skorohod were the most important pioneers, even though the concepts had already roots in earlier work by Langevin and Bernstein) and stochastic analysis. On the other hand strong motivations from hydrodynamics (turbulence theory) and especially from quantum field theory urged the development of higher space-time dimensional analogues of measures on paths spaces, i.e.  measures associated with processes depending, in addition to time, on space variables. In the west pioneering work in this direction  was achieved by K.O. Friedrichs \cite{Frie}, I. Segal \cite{Se}, L. Gross \cite{Gro2} and in the east by Gelfand and his school, see, e.g., \cite{GelVi,Pro}. A further strong momentum came from work by E. Nelson, see, e.g. \cite{Ne2,ReSi}, (influenced by K Symanzik and physicists like Nagano and  Schwinger) on a Euclidean approach to relativistic quantum field theory, which permitted to construct models of low space dimensional quantum fields essentially by probabilistic and statistical mechanical methods, see, e. g. \cite{Si,GlJa,BSZ,Alb,AFHKL,AKKR1,HKPS}. Later on an important connection with SPDE's was established, see, e.g. \cite{Alb,Ne2, AlHK74,AR89,DaPra06,DaP,DaPra,Dem,Doo,Doo2,Ikeda,PeZa,HKPS,HiS,RudigerWu2001,AZe,Streit,DipersioMastrogiacomo2011}. In all these constructions and applications, probability measures on infinite dimensional spaces are constructed by starting from projective systems of finite dimensional probability measures. A Kolmogorov's type construction, see, e.g., \cite{AMaLausanne} yields then the probability measure on a suitable projective limit, on which it is indeed $\sigma$-additive.
The relevant quantities associated with the processes and fields are then expressed as integrals with respect to these probability measures (with support consisting, in general, of singular or generalized functions). The connection of such integrals with equations of quantum mechanics resp. quantum field theory or, more generally, hyperbolic problems are, rather indirect, via ``analytic continuation'', whenever possible (this is rarely the case if space times are not flat, by the way).\\
As discovered by Feynman (1942-1948), on the other hand (anticipated to some extent by Wentzell (1905) and Dirac (1933)), to express quantum mechanical quantities, like solutions of the Schr\"odinger equation or like time evolved position operators, one would need oscillator-type integrals on spaces of paths or fields. This leads to the problem of the construction of ``Feynman path integrals'' or more generally oscillatory integrals in infinite dimensions.\\
The finite dimensional oscillatory integrals should somehow be imbedded in this construction. One can then also expect to obtain, as it was for the finite dimensional oscillatory integrals, a good setting for a method of stationary phase in infinite dimensions. This program has been initiated by Yu. Daletskii \cite{Da,DaFo} and K. It\^{o} \cite{Ito}, and continued systematically in \cite{AlHKMa,AlHK,ELT,AlBr1,AB,AlBre,AlBrHa,GuatteriMazzucchi2002,GuatteriMazzucchi2003,GuatteriMazzucchi2004,HahnSengupta2004,AlJoPaSca,AlKhSmo, Mazzucchi2004, Mazzucchi2004a,AMBull,AlMa05,Mazzucchi2005b,Mazzucchi2006,Mazzuchi2007,Mazzucchi2009,Mazzucchi2009a,Mazzucchi2011,Mitoma2009,AlSe,Bor,CarNu,DeFPS,ELT,
Exn,Fuj,GroStreVo,He,Ich, Kal2, KalKanKar, Kol1,Kol2,LeuSch,Man,Ma2008,Ne,Roe,Tru,Zas}, see also \cite{Alb,AlMaScolarpedia} and references therein.\\
A connection with certain infinite dimensional distributions has also been achieved \cite{HKPS,Hida} (see also \cite{CaDeW,Ma11}).\\
The present paper aims at unifying the constructions of probabilistic and oscillatory infinite dimensional integrals by starting from corresponding finite dimensional integrals and considering their limit, stressing the relevant continuity properties. In the case of integrals with respect to probability (or bounded complex) measures (i.e. of integrals of the absolutely convergent type) one has the kind of usual continuity given by $\sigma$-additivity in the case of oscillatory integrals (non absolutely convergent type) one has continuity in the sense of suitable norms on the space of functions to be integrated.\\
More precisely our unified approach starts with the construction of projective systems of  functionals. Since in the probabilistic – functional analytic literature there are various concepts of projective systems, in particular according to the chosen underlying measure resp. topological space structure, we present our approach in great generality, specializing only ``au fur et \`a mesure'', when it becomes necessary.\\
In section 2 we present our concept of projective systems and of perfect inverse systems for general  spaces and functionals, then we illustrate it by two examples which might be more familiar to analysis and probability oriented readers.
In section 3 we discuss extensions of projective systems of linear functionals. We introduce cylindrical functions and distinguish between different extensions of projective systems of functionals definied on them. In Theor. 1 we characterize the minimal extension and show that it can also be described as the direct (or inductive) limit of a suitable direct system (Theor. 2). The maximal extension is also described, and a necessary and sufficient condition for it is identified ( Prop.1).\\
In Example 3 a situation is described where one does not have uniqueness of extensions, i.e. the action of the linear functional on the projective limit is not identified by just looking at its action on cylinder functions. \\
In Sect. 4 we study the continuity of extensions of projective systems of functionals.  A topology on cylindrical functions is introduced and then it is shown that the minimal extension of the projective system of functionals is closable, whenever it is continuous (on its natural domain, Theor.3).
We also relate compatible topologies to the properties of the closures of the graphs associated with the minimal extension, and show that non compatible topologies lead to non comparability of the closures of the graphs of the minimal extension.\\
In Sect. 5 we study projective systems of signed or complex bounded measure spaces and compare them with projective systems of probability measure spaces, on one hand, and on the other hand, we describe what goes wrong when the measurable spaces have an unbounded measure on them ( measures with infinite variation). We recall the notion of a compatible family of complex measures on a projective system with respect to a system of projections. We also recall Kolmogorov's type existence theorems in the case of probability measures. In the case of complex measures we formulate a condition of uniform boundedness of variation (Theorem 4) as a necessary condition for the family to arise from projections of a complex measure on the projective limit. We also present E. Thomas' extension of Kolmogorov's existence theorem to the bounded complex measures case, in a topological setting, with a projective family of Radon measures. We underline that many interesting cases are not covered by this extension, a problem which we take up again in   Section 6.\\
Still in Section 5 we discuss (in Sect. 5.1) the special case where the projective family is made of product measures, and we indicate already here cases where Thomas' result can not be applied.\\
In Sect. 5.2 we discuss the case where the projective family is constructed using semigroups of complex kernels. We associate to them pseudo processes (Markov resp. sub- Markov processes in the probabilistic case). We then specialize to vector spaces and consider the particular case of translation invariant complex kernels given in terms of the fundamental solutions of parabolic higher order or hyperbolic PDEs. As a consequence of Thomas' theorem we show that the projective family constructed by such kernels does not have a projective extension given by an integral with respect to a complex measure on the projective limit space.\\
In Sect.5.3 we study projective systems given by oscillatory integrals (the typical case of Feynman path integrals). We observe that they do not have any projective extension on the projective space $\R^\N$ given by (Lebesgue-type) improper integrals with respect to complex measures. This motivates looking for other types of projective extensions (i.e., extensions not in the form of integrals with respect to complex measures). \\
In Sect. 6 we present several examples of such types of extensions. The first one uses Fourier transforms on a real separable infinite dimensional Hilbert space to construct a projective family, and is based on an infinite dimensional version of Fresnel integrals and Parseval's formula, originated in \cite{AlHK,AlHKMa}. The continuity property of the projective limit functional is made explicit. A second example considers a projective systems of complex kernels associated with the fundamental solution of higher order PDEs of hyperbolic or parabolic type (of the form \eqref{PDE-p} below). Again a continuous extension of the projective system is constructed. In Theor. 6 its domain and action are described. 
In Theor. 7 a general Feynman-Kac formula is expressed in terms of the projective limit  discussed in Theor. 6. This formula solves the Schr\"odinger resp. higher-order parabolic / hyperbolic evolution equations involving integer powers of the Laplacian, perturbed by a potential.\\
A further example of projective extension which is not of the type of integration with respect to a complex measure type is provided in Theor. 8. This and related constructions have been used to solve Schr\"odinger and heat equations with a potential which is the sum of a quadratic potential and a magnetic potential, resp. a potential of polynomial growth. \\
In Sect. 7 we shortly present some conclusions.


\section{Projective systems of functionals}

Let us recall some classical definitions (see, e.g, \cite{BouT,Bou,Bau,Bau5,Loe,PfP,Schwartz}) which we will use all along the paper.

\begin{definition} \label{def1}
A semi-ordered set $A$ with partial ordering $\leq$ is called a {\em directed set} if for any $J,K\in A$, there exists $R\in A$ such that $J\leq R$, $K\leq R$
\end{definition}

\begin{definition}\label{def2}
Let us consider a family $\{E_J\}_{J\in A}$ of (non-empty)  sets $E_J$  labelled by the elements of a 
 non-empty directed  set $A$, called index set. Let us assume that for any $J,K\in A$, $J\leq K$, there exists a surjective map $ \pi^K_J:E_K\to E_J$  such that $\pi_K^K$ is the identity on $E_K$ and for  all $J\leq K\leq R$, $R\in A$,  one has $\pi^R_J=\pi^K_J\circ \pi^R_K$  (``consistency property''). Such a family $\{E_J,\pi^K_J\}_{J,K\in A}$ is called a {\em projective (or inverse) family of sets}.\\
The projective family $\{E_J,\pi^K_J\}_{J,K\in A}$ is called {\em topological} if each  $E_J$, $J\in A$, is a topological space and the maps $ \pi^K_J:E_K\to E_J$, $J\leq K$, are continuous.
\end{definition}

\begin{example}\label{example1}
A special case of this situation is given when we have a family  $(E_t^0)_{t\in I}$ of non empty sets  (``spaces'') labelled by the elements of a non-empty set $I$ and we take $A=\Fo(I)$ to be the totality of non empty subsets of  $I$ with a finite number of elements. The partial ordering  $\leq $ in $\Fo(I)$ is then defined in terms of  inclusions of sets, i.e. 
if $J,J'\in \Fo(I)$ and $J\subseteq J'$ then $J \leq J'$. Define  for each $J\in \Fo(I)$ 
$$E_J=\Pi_{t\in J}E_t^0,$$
i.e. $E_J$ is the ``Cartesian product of the $E^0_t$ along $J$''. We also say that $E_J$ is a product space (constructed from the $E^0_t$). We then have $E_{\{t\}}=E_t^0$. \\
If $J\leq K$, $J,K\in \Fo(I)$, then we can take $\pi^K_J$ to be the canonical projection mapping $\pi^{0K}_J$ from $E_K$ to $E_J$, given by 
$$\pi^{0K}_J \omega_K=\omega_J,$$
with $\omega_K:=(\omega _t)_{t\in K}$ understood as the element in $E_K$ with component $\omega_t$ in $E^0_t$. \\
If the spaces $(E_t^0)_{t\in I}$ are topological, then for any $J\in \Fo(I)$ we endowe the product space $E_J=\Pi_{t\in J}E_t^0$ with the product topology, i.e. the coarsest topology for which all the projections $\pi_{t'}:\Pi_{t\in J}E_t^0\to E_{t'}^0$, $t'\in J$, are continuous.
For $J,K\in\Fo(I)$, $J\leq K$, the projection $\pi^K_J$ from $E_K$ to $E_J$ is then continuous. 
By constructions $(E_J,\pi^K_J)_{J,K\in \Fo(I)}$ is a topological projective family.
\end{example}

\begin{example}\label{example2}
Let $(\Hi, \langle \; ,\; \rangle)$ be a real  Hilbert space. Let us consider the set $A$ of orthogonal projection operators $P$ onto   finite dimensional subspaces $P(\Hi)$ of $\Hi$. $A$ is a directed set, with partial order given by $P\leq Q$, $P,Q\in A$, iff $P(\Hi)\subseteq Q(\Hi)$.
Set $E_P:=P(\Hi)$. For $P\leq Q$,  let $\pi^Q_P:Q(\Hi)\to P(\Hi)$ be the projection $\pi^Q_P:=P|_{Q(\Hi)}$, where $P|_{Q(\Hi)}$ stands for the restriction of $P$ to $Q(\Hi)$. One has that  the  family $(E_P,\pi^Q_P)_{P,Q\in A}$ forms a projective family of sets. \\
In fact, if the finite dimensional vector spaces $E_P$ are endowed with the natural topology, the projections $\pi^Q_P$ are continuous and $(E_P,\pi^Q_P)_{P,Q\in A}$ forms a topological projective family of sets.
\end{example}

Given a projective family $\{E_J,\pi^K_J\}_{J,K\in A}$,
 we shall consider complex-valued functions $f_J$ defined on $E_J$, for any $J\in A $. $f_J$ is thus a map from $E_J$ into $\C$. We shall call $\hat E_J$ the space of all such functions on $E_J$.

Let $f_J\in \hat E_J$, $J\in A$. For any $K\in A$ with $J\leq K$ we can define the extension $\Ep_J^K(f_J)$ of $f_J$ to $E_K$ as the function belonging to $\hat E_K$ given by:
$$\Ep_J^K(f_J)(\omega_K):=f_J\big(\pi_J^K(\omega_K)\big), \qquad \omega_K\in E_K.$$
If $\{E_J,\pi^K_J\}_{J,K\in A}$ is a topological projective family and $f_J\in \hat E_J$ is a continuous function, then for any $K\in A$ with $J\leq K$, the extension $\Ep_J^K(f_J)$ is a continuous function on $E_K$.	

Let us now consider linear maps from subsets $\hat E_J^0\subseteq\hat E_J$ of $\hat E_J$ to $\C$, called functionals. $L_J$ is thus such a functional  if $L_J$ associates to a function $f\in\hat E_J^0$ a complex number $L_J(f)$ and for any $\alpha,\beta\in \C$, $f,g\in \hat E_J^0$ the following holds:
$$L_J(\alpha f+\beta g)=\alpha L_J(f)+\beta L_J(g) .$$ $\hat E_J^0$ is called domain of $L_J$, the set $\{L_J(f)\}_{f\in \hat E_J^0}$ is called range of $L_J$. We shall call $Map(\hat E_J)$ the family of all such functionals.

For $J\leq K$, let us define the map $\hat \pi^K_J:Map(\hat E_K)\to Map(\hat E_J)$ as the transport of any functional  $L_K\in Map(\hat E_K) $ induced by the map $\Ep^K_J$ from $\hat E_J$ to $\hat E_K$, given by:
\begin{equation}\label{eq3}
\hat \pi^K_J(L_K)(f_J):=L_K\left( (\Ep_J^K(f_J)\right), \qquad L_K\in Map(\hat E_K),
\end{equation}
where the domain of $\hat \pi^K_J(L_K)$ is given by
$$Dom(\hat \pi^K_J(L_K))=\{f_J\in \hat E_J, \, | \Ep_J^K(f_J)\in \hat E_K^0\}.$$

Let us consider a family of functionals $\{L_J,\hat E_J^0\}_{J\in A}$ labelled by the elements of an index set $A$.

\begin{definition}\label{def3}
We call the family  $\{L_J,\hat E_J^0\}_{J\in A}$ a {\em projective system of functionals} 
if for all $J,K\in A$ with $J\leq K$ the projective (or coherence or compatibility) conditions hold
\begin{equation}\nonumber
\Ep_J^K(f_J)\in \hat E_K^0, \qquad \forall f_J\in \hat E_J^0,
\end{equation}
\begin{equation}\label{eq4}
\hat \pi^K_J(L_K)(f_J)=L_J(f_J),\qquad \forall f_J\in \hat E_J^0.
\end{equation}
\end{definition}

\begin{definition}\label{def4}
Consider the general setting as in definition \ref{def2}.
We define the {\em projective (or inverse) limit } $E_A:=\varprojlim E_J$ of the projective family  $\{E_J, \pi^K_J\}_{J,K\in A}$ as
 the following subset of the direct (or Cartesian) product of the family $\{E_J\}_{J\in A}$:
$$E_A:=\{(x_J)\in \Pi_{J\in A}E_J,\, |\,  x_J=\pi^K_J(x_K)\; \hbox{\rm  for all }J\leq K,\, J,K\in A\}$$
\end{definition}

The space of all complex valued functions on $E_A$ will be denoted with $\hat E_A$.

\begin{Remark}
In the case of example \ref{example1} one has that $\varprojlim E_J$ is isomorphic to the product space $\Pi_{t\in I}E^0_t$.\\
In the case of example \ref{example2}, with $\Hi$ separable, we have $\Hi\subset E_A$, $\Hi\neq E_A$.  Indeed $\Hi$ is strictly included in the projective limit of its finite dimensional subspaces (in the sense of Def \ref{def4}).
To see this, let us choose an orthonormal base of $\Hi$, then using this base we have that $\Hi$ is isomorphic to $ l_2$ and the finite  dimensional subspaces of $\Hi$ can be identified with the sequences $\{x_n\}\in l_2$ with only a finite number of terms different from 0. Now the projective limit of the family of finite dimensional subspaces of $l_2$ in the sense of Def. \ref{def4} is the space $\R^\N$ of all sequences, that strictly includes $l_2$.
\end{Remark}

Let $\tilde E:=\Pi_{J\in A}E_J$.
Let $\tilde \pi_J:\tilde E\to E_J$ be the coordinate projection of $\tilde E$ into $E_J$, so that if $\tilde \omega=\{\omega_J, J\in A\}\in \tilde E$ then $\tilde \pi _J(\tilde\omega)=\omega_J$. Let $\pi_J:=\tilde \pi_J|E_A$ be the restriction of $\tilde \pi _J$ to $E_A$. 
One has that for any $J,K\in A$, with $J\leq K$ \begin{equation}\label{rel-pi}\pi_J=\pi^K_J\circ \pi_K.\end{equation}

If $(E_J,\pi_J^K)_{J,K\in A}$ is a topological projective family, then $E_A=\varprojlim E_J$ (as defined by Def. \ref{def4}) will be  endowed with the coarsest topology making all the projection maps $\pi_J:E_A\to E_J$ continuous. This is also called {\em initial } or {\em inductive } topology \cite{BouT}.

\begin{Remark}
Given a general projective family   $\{E_J, \pi^K_J\}_{J,K\in A}$, two problems may occur:
\begin{enumerate}
\item Even if $E_J\neq \emptyset$ for any $J\in A$, it might happen that $E_A=\emptyset$. See, e.g. \cite{Gha}.
\item Even if all the projections $(\pi^K_J)_{J,K\in A}$ are surjective, the maps $\pi_J:E_A\to E_J$ may fail to be surjective. See, e.g. \cite{Lef}.
\end{enumerate}
\end{Remark}

\begin{definition}
A projective family $\{E_J, \pi^K_J\}_{J,K\in A}$ is called {\em perfect inverse system} if for all $J\in A$, $x_J\in E_J$, there exist an $x\in E_A$ (with $E_A$  as in Def. \ref{def4}) such that $x_J=\pi_J x$. In this case all the projections are surjective.
\end{definition}

\begin{Remark}
In the terminology of \cite{Boc} (see also \cite{Bo2}) a perfect inverse system  $\{E_J, \pi^K_J\}_{J,K\in A}$ is called {\em simply maximal}.
\end{Remark}
 
One can easily verify that the projective families presented in examples \ref{example1} and \ref{example2} are perfect inverse systems. In fact in the following we shall always assume, unless otherwise stated, that the inverse systems we are considering are perfect.\\

Given a function $f_J \in \hat E_J$, $J\in A$, it can be extended to a function $\Ep_J^A f_J:=\Ep_J^A( f_J)$ on the projective limit $E_A$ in the following way
$$\Ep_J^Af_J(\omega):=f_J(\pi_J\omega), \qquad \omega \in E_A.$$
By eq \eqref{rel-pi},  the extension maps $\Ep_J^A:\hat E_J\to \hat E_A$ satisfy the following condition for any $J,K\in A$, with $J\leq K$:
 \begin{equation}\label{rel-Ep}\Ep^A_J=\Ep^A_K\circ \Ep_J^K.\end{equation}

If $(E_J,\pi_J^K)_{J,K\in A}$ is a topological projective family, then all the extensions  $\Ep_J^K:\hat E^J\to \hat E_A$ and $\Ep_J^A:\hat E^J\to \hat E_A$ map continuous functions into continuous functions.\\

Given a projective system of functionals $\{L_J, \hat E^0_J\}_{J\in A}$ (in the sense of Def \ref{def3}), 
we shall write $\Cil$ for the family
 $$\Cil=\cup_{J\in A}\Ep^A_J(\hat E_J^0).$$
The functions belonging to $\Cil$ will be called {\em cylindrical (or cylinder) functions}.
In particular, in the case of example \ref{example1} or example \ref{example2}, the cylindrical functions depend only on a finite number of variables. The following lemma states the injectivity of the extension map $\Ep^A_J:\hat E_J\to \hat E_A$.

\begin{lemma}\label{lemmaunique}
Let $\{E_J, \pi^K_J\}_{J,K\in A}$ be a perfect inverse system. Then for any $J\in A$, $f,g\in \hat E_J$:
$$f=g \quad \Leftrightarrow \quad \Ep_J^Af=\Ep_J^Ag.$$
\end{lemma}

\begin{proof}
The implication $f=g \Rightarrow \Ep_J^Af=\Ep_J^Ag$ is trivial.\\
To prove the converse $\Ep_J^Af=\Ep_J^Ag \Rightarrow f=g$, let us assume that there exists an $x_J\in E_J$ such that $f(x_J)\neq g(x_J)$. Let us consider an element $x\in E_A$ such that $\pi_Jx=x_J$. Such an element exists because the projections $\pi_J$ are surjective by assumption. We thus have by the definition of $\Ep_J^Af(x)$, resp. $\Ep_J^A(x)$,  for $f\in \hat E_J$, resp. $g\in \hat E_J$ $$\Ep_J^Af(x)=f(x_J)\neq g(x_J)=\Ep_J^Ag(x)$$ and thus $\Ep_J^Af\neq\Ep_J^Ag $.
\end{proof}


\section{Extensions of a projective system of functionals}

Given a projective system of functionals $\{ L_J,\hat E^0_J\}_{J\in A}$, we shall denote by $\Cil_0\subset \Cil$ the subfamily of cylindrical functions consisting of those cylindrical functions which are obtained by  extensions  $\Ep_J^Af_J$ of $f_J \in \hat E^0_J$ to the projective limit $E_A$, i.e.:
$$\Cil_0:=\cup_{J\in A}\Ep^A_J(\hat E_J^0)=\{f\in \Cil\, |\, f=\Ep_J^A f_J, \hbox{ for some } J\in A, f_J \in \hat E^0_J   \}.$$ 
 
\begin{definition}\label{def5}
A {\em projective extension} $(L,D(L))$ of a projective system of functionals $\{L_J, \hat E_J^0\}_{J\in A}$  is  a functional $L$ with domain $D(L)\subseteq \hat E_A$ ($\hat E_A$ being the complex-valued functions on $E_A$), such that 
\begin{itemize}
\item  $\Cil_0 \subseteq D(L)$, 
\item $L(\Ep_J^A f_J)=L_J(f_J)$, for all $f_J\in \hat E_J^0$.
\end{itemize} 
\end{definition}

\begin{theorem}\label{th-minimalextension}
Let $\{E_J, \pi^K_J\}_{J,K\in A}$ be a  perfect inverse system and let $\{L_J, \hat E_J^0\}_{J\in A}$ be a projective system of functionals. Then a projective extension  $(L,D(L))$ exists. It is the functional defined by:
 \begin{eqnarray}\nonumber
D(L)&:=& \{f\in \hat E_A,\,|\, \hbox{ \rm there exists  } J\in A, f_J\in \hat E_J^0, \,f=\Ep_J^Af_J\}=\Cil_0\\
L(f) &:= & L_J(f_J), \qquad f=\Ep_J^Af_J, \, f_J\in \hat E_J^0. \nonumber
\end{eqnarray}
This functional is  ``minimal'' in the sense that any other extension $( L',D( L'))$ of the  projective system of functionals $\{L_J, \hat E^0_J\}_{J\in A}$ is such that
$D(L)\subseteq D(L')$ and
 $L'(f)=L(f)$ for all $f\in D(L)$. 
\end{theorem}

\begin{proof}
The functional $L$ is well defined. Indeed let us consider a cylindrical function in the restricted class $\Cil_0$, i.e. $f\in\Cil_0$ that can be obtained both as the extension of $f_J\in \hat E_J^0$ and of $f_K\in \hat E_K^0$, i.e. $f=\Ep_J^Af_J=\Ep_K^Af_K$, $J,K\in A$. By the very definition of directed set, there exists an $R\in A$, with $J\leq R$ and $K\leq R$. Let $f_R:=\Ep_J^Rf_J$ and $g_R:=\Ep_K^Rf_K$. One can easily see that $\Ep_R^Af_R=\Ep_R^Ag_R=f$. Thus, by lemma \ref{lemmaunique} $f_R=g_R$ and $L_R(f_R)=L_R(g_R)$. On the other hand, by the projectivity condition \eqref{eq4} it follows that:
$$L_J(f_J)=L_R(\Ep_J^Rf_J)=L_R(f_R)=L_R(g_R)=L_R(\Ep_K^Rf_K)=L_K(f_K).$$ 
Thus $L$ is thus unambiguosly defined by $L(f)=L_J(f_J)=L_k(f_K)$.
\end{proof}

\begin{definition}\label{def6}We shall call the projective extension  $(L,D(L))$ described in theorem \ref{th-minimalextension} {\em the minimal extension} of the projective system of functionals. In the following we shall denote the minimal extension by $(L_{min},D(L_{min}))$. We shall write $L_{min}=\varprojlim L_J$.
\end{definition}

The domain of the minimal extension of a projective system of functionals $\{L_J, \hat E^0_J\}_{J\in A}$ can be described in terms of  the direct or inductive limit of a suitable  direct system. We recall here the definition of direct system and direct limit, see also \cite{Rot}. 

\begin{definition} \label{def8} Let $(A,\leq)$ be a directed set. Let $\{ E_J\}_{J\in A}$ be a family of (non empty) sets indexed by the elements of $A$, endowed with a family of maps $F_{JK }: E_J\to E_K$, for $J\leq K$, such that
\begin{itemize}
\item  $F_{JJ}$ is the identity of $E_J$ for any $J\in A$, 
\item  $F_{KR}\circ F_{JK}=F_{JR}$ for all $J\leq K\leq R$.
\end{itemize}
Then the pair $(  E_J,F_{JK})_{J,K\in A}$ is called a {\em direct system} on $A$.\\
The {\em direct (or inductive) limit} of the direct system $(  E _J,F_{JK})_{J,K\in A}$ is denoted by $ \varinjlim  E_J$ and defined as the disjoint union $\cup_J E_J$ modulo an equivalence relation $\sim$:
$$\varinjlim  E_J=\cup_J E_J/\sim$$
where, if $\omega_J\in  E_J$ and $\omega_K\in  E_K$, then $\omega _J\sim \omega_K$ if there is some $R\in A$ with $J\leq R$, $K\leq R$, and $F_{JR}(\omega_J)=F_{KR}(\omega_K)$.
\end{definition}

A family of maps $F_J:E_J\to  \varinjlim  E_J$ naturally arises, where $F_J:E_J\to  \varinjlim  E_J$ maps each element of $E_J$ into its equivalence class. Further $F_J=F_K\circ F_{JK}$ for all $J,K\in A$, with $J\leq K$.\\
If the sets $(E_J)_{J\in A}$ are topological spaces and the maps $(F_{JK})_{J,K\in A}$ are continuous, the family $(  E_J,F_{JK})_{J,K\in A}$ is called a {\em topological direct system}. Its direct limit is the space $\varinjlim  E_J$ endowed with the finest topology making all the maps $F_J:E_J\to  \varinjlim  E_J$ continuous.

Direct and projective limits are dual in the sense of categoy theory, see, e.g., \cite{McLane}.\\
An application of the concept of Def \ref{def8} is found in considering the
 family of sets $\{\hat E^0_J\}_{J\in A}$ of Def. \ref{def3}, labelled by the elements of the directed set $A$, endowed with the extension maps 
$$\Ep^K_J:\hat E^0_J\to \hat E^0_K,\qquad J\leq K, J,K\in A.$$
This family forms a direct system of sets, and its {\em direct} or {\em inductive} limit in the sense of Def. \ref{def8} is equal to the restricted set $\Cil_0$ of cylindrical functions (defined at the beginning of this section), as stated in the following theorem.

\begin{theorem}
Under the same assumptions of Theorem \ref{th-minimalextension}, the domain of the minimal extension is given by the direct limit of the direct system $\{\hat E^0_J,\Ep_J^K \}_{J,K\in A}$:
$$D(L_{min})=\varinjlim \hat E_0^J.$$ 
\end{theorem}

\begin{proof}
We prove a 1 to 1 correspondence between  $D(L_{min})=\Cil_0$ and $\varinjlim \hat E_0^J$.
Let $g\in \varinjlim \hat E_0^J$, i.e., by Def. \ref{def8}, $g$ is an equivalence class of functions $[f_J]$ such that for any two elements $f_J$ and $f_K$ belonging to $g\equiv [f_J]$ there is an $R\in A$, with $J\leq R$ and $K\leq R$, such that $\Ep_J^R(f_J)=\Ep_K^R(f_K)$. The functions belonging to this equivalence class $g$ define a unique cylindrical function. Indeed let $f_J,f_K\in g$, and let $\Ep_J^Af_J$ and $\Ep_K^Af_K$ be the associated cylindrical functions. One has that, for any $x\in E_A$ (with $E_A$ the projective limit of the $(E_J,\pi_J^K)$, $J,K\in A$) we have $\Ep_J^Af_J(x)=\Ep_K^Af_R(x)$. Indeed let $R\in A$, with $J\leq R$ and $K\leq R$, such that $\Ep_J^R(f_J)=\Ep_K^R(f_K)$ (such an $R$ exists as $f_J$ and $ f_K$ belong to the same equivalence class $g$). Then by eq \eqref{rel-Ep} $$\Ep_J^Af_J(x)=\Ep_R^A\circ\Ep_J^Rf_J(x)=\Ep_R^A\circ\Ep_K^Rf_K(x)=\Ep_K^Af_K(x).$$
Conversely, under the assumption that $\{E_J, \pi^K_J\}_{J,K\in A}$ is a  perfect inverse system, to any cylindrical function belonging to $\Cil_0$  it is possible to associate an element of $\varinjlim \hat E_0^J$. Indeed let $f\in\Cil_0$ of the form $f=\Ep_J^Af_J$, $J\in A$, $f_J\in \hat E^0_J$, and let $[f_J]\in \varinjlim \hat E_0^J$ be the equivalence class of $f_J$. Let us assume that $f$ can be also be obtained as $f=\Ep_K^Af_K$, for some $K\in A$, $f_K\in \hat E^0_K$, then we shall show that $[f_K]=[f_J]$, i.e. there exists an $R\in A$, with $J\leq R$ and $K\leq R$ such that $\Ep_J^Rf_J=\Ep_K^Rf_K$. Indeed let $R\in A$ with $J\leq R$ and $K\leq R$ and let us consider the functions $f_R, f'_R\in \hat E_R^0$ of the form $f_R:=\Ep_J^Rf_J$ and $f'_R:=\Ep_K^Rf_K$. By eq \eqref{rel-Ep} it follows that 
$\Ep_R^Af_R=\Ep_R^Af'_R$, indeed $\Ep_R^Af_R=\Ep_R^A\Ep_J^Rf_J=\Ep_J^Af_J=f$ and analogously $\Ep_R^Af'_R=\Ep_R^A\Ep_K^Rf_K=\Ep_K^Af_K=f$, and by lemma \ref{lemmaunique} it follows that $f_R=f'_R$ (since we assumed that $(E_J,\pi_J^K)$, $J,K\in A$, is a perfect inverse system).
\end{proof}


We can ask ourselves whether there exist a ``maximal'' extension   of a projective system of functionals $\{L_J, \hat E^0_J\}_{J\in A}$, i.e. a functional $(L_{max},D(L_{max}))$ such that for any extension $(\tilde L,D(\tilde L))$ of  $\{L_J, \hat E^0_J\}_{J\in A}$ one has that
\begin{eqnarray}\nonumber
 & &D(\tilde L)\subseteq D( L_{max})\\
& &  L_{max}(f)=\tilde L(f),\qquad \forall f\in D(\tilde L).\nonumber
\end{eqnarray}
The problem is strictly connected to the uniqueness property of the extensions of a projective system. Indeed if there are two  extensions $(L,D(L))$ and $(L',D(L'))$ such that there exists an element $f\in D(L)\cap D(L')$, with  $L(f)\neq L'(f)$, then it is not possible to construct an extension $\tilde L$ of both $L$ and $L'$, as $\tilde L $ would be ambiguously defined on the element $f$. The converse is also true, as is stated in the following proposition.

\begin{proposition}\label{prop1}
Let $\{L_J, \hat E^0_J\}_{J\in A}$ be a projective system of functionals and let $\Fo=\{ (L,D(L)\}$ be a non void family of projective extensions of $\{L_J, \hat E^0_J\}_{J\in A}$. The family $\Fo$ has a  maximal element if and only if it satisfies the following ``uniqueness property'':\\
whenever two extensions $(L,D(L)),(L',D(L'))\in \Fo$ have an element  $f\in D(L)\cap D(L')$, one has that $L(f)=L'(f)$.
\end{proposition}

\begin{Remark}
Under the assumption that $(E_J, \pi_{J}^K)_{J,K\in A}$ is a perfect inverse system, it is always possible to construct a non void family of  projective extensions of $\{L_J, \hat E^0_J\}_{J\in A}$, as the minimal extension exists by Theorem \ref{th-minimalextension}.
\end{Remark}

\begin{proof} (of Proposition \ref{prop1}) It is trivial, as  remarked before the statement of Prop. \ref{prop1}, that the uniqueness property is a necessary condition for the existence of a maximal element in the family $\Fo$. \\
To prove the sufficiency, let us remark that the family  $\Fo$
is a non void  partially ordered set, with partial order $\leq$,  where by definition two extensions  $(L,D(L)),(L',D(L'))\in \Fo$ satisfy  $(L,D(L))\leq (L',D(L'))$ if $D(L)\subset D(L')$ and $L(f)=L'(f)$ $\forall f \in D(L)\cap D(L')$). Moreover any chain $\{L_\alpha,D(L_\alpha)\}_\alpha \subset \Fo$ has an upper bound  $(\tilde L;D(\tilde L))$ given by
$$  D(\tilde L))=\cup_\alpha D(L_\alpha)\subset \hat E_A, \qquad L(f)=L_\alpha(f), \; f\in D(L_\alpha),$$ 
(the set inclusion being clear since all $D(L_\alpha)$ are in $\hat E_A$).
By Zorn's lemma (e.g. \cite{HewStr} pag 13-14) we  deduce then that the set of all possible extensions of  a projective system of functionals has a maximal element.
\end{proof}

Given two extensions $(L,D(L))$ and $(L',D(L'))$ of a projective system of functionals $\{L_J, \hat E^0_J\}_{J\in A}$, in general it is not always true that they coincide on the intersection of their domains, i.e. that for any $f:E_A\to \C$ such that $f\in D(L)\cap D(L')$, one has that $L(f)=L'(f)$. Indeed let us consider the following counterexample.\par

\begin{example}\label{es-not-uniqueness}
Let us consider example \ref{example1} in the particular case where $I=\N$, hence
 $A=\Fo(\N)$ is the totality of non empty finite subsets of $\N$. Given an element $J\in A$, let $E_J$ be given by
$E_J=\R^J$.
We also identify $E_A$ with   $E_A=\R^\N$, the set of all real valued sequences.\\
Let $\mu$ denote a probability measure on $\R$ and let $(L_J,\hat E_J^0)$ be given by
\begin{eqnarray}\nonumber
\hat E_J^0&=& B_b(E_J)\\
L_J(f) &=& \int _{E_J}f\ d\Pi_{n\in J}\mu, \qquad f \in \hat E_J^0 \nonumber
\end{eqnarray}
where $B_b(E_J)$ is the set of Borel bounded functions on $E_J$ and $\Pi_{n\in J}d\mu$ denotes the measure on $E_J$ obtained as the product of $|J|<\infty$ copies of the probability measure $\mu$ ( $|J|$ denoting as before the cardinality of the set $J\in A$).
By (a special case of) Kolmogorov's existence theorem \cite{Bau} there exists a  probability measure $\mu_A$ on the $\sigma-$algebra generated by the cylindrical sets in  $E_A=\R^\N$ obtained as the product $\Pi_{n\in \N}\mu$. Moreover, if $\mu$ is e.g. the standard centered Gaussian measure on $\R$, the subset $l_2\subset \R^\N $, $l_2=\{\omega \in \R^\N,\,|\,\sum _n|\omega_n|^2<\infty\}$, has  zero $\mu_A=\pi_{n\in \N}\mu$-measure \cite{Bog2}.  \\
Let us consider the following two different extensions $(L,D(L))$ and $(L',D(L'))$ of the projective system of functionals $\{L_J, \hat E_0^J\}_{J\in A}$. The functional $L$ is defined by:
\begin{eqnarray}\nonumber
D(L)&=& B_b(\R^\N)\\
L(f) &=& \int _{\R^\N}f(\omega)\mu_A(d\omega) ,\quad f\in D(L)\label{L-1}
\end{eqnarray}
where $B_b(\R^\N)$ denotes the Borel bounded functions over $\R^\N$. \\
Denote by $\pi_n:\R^\N\to \R^n$ the projection of a sequence $\omega \in \R^\N$ to the sequence $\pi_n \omega$  given by 
$$(\pi_n \omega)_m=\left\{ \begin{array}{ll}
\omega_m &\qquad \hbox{\rm if $m\leq n$ }\\
0 & \qquad\hbox{\rm otherwise. },\\
\end{array}\right.$$
 Let us consider the functional  $L'$ given by:
\begin{eqnarray}\nonumber
D(L')&=& \{f:\R^\N\to \C,\, |\, \exists\lim_{n\to \infty}\int f(\pi_n\omega)\mu_A(d\omega)\}, \\
L'_J(f) &=& \lim_{n\to \infty}\int f(\pi_n\omega)\mu_A(d\omega) \label{L-2}
\end{eqnarray}
Let us now  consider the function $f:\R^\N\to \C$ defined by
$$f(\omega)=\left\{ \begin{array}{ll}
0 &\qquad \hbox{\rm if $\omega$ has an infinite number of non-vanishing components }\\
1 & \qquad\hbox{\rm otherwise }\\
\end{array}\right.$$
$f$ is the characteristic function of the set $$B=\{\omega\in \R^N,\,| \hbox{\rm  $\omega$ has a finite number of non-vanishing components }\}.$$ One can easily verify that $B$ is a measurable subset of $\R^\N$, indeed it belongs to the $\sigma$-algebra generated by the cylindrical sets, as:
$$B=\cup_{N\in \N}\cap_{n>N}B_n$$
where
$$B_n:=\{\omega\in \R^N\,|\, \omega_n=0\}.$$
Then one has, for $\mu_A$ the above Gaussian measure
\begin{eqnarray}\nonumber
L(f)&=&\mu^A(B)\\
&\leq& \mu^A(l_2)=0\nonumber
\end{eqnarray}
while
$$L'(f)=\lim_{n\to \infty}\int _{\R^n}f(\pi_n\omega)\mu _A(d\omega)=\lim_{n\to \infty} 1=1.$$
In other words, without asking for additional properties of the extended functionals, the extensions $(L,D(L))$ are not uniquely determined by the projective system $\{L_J, \hat E^0_J\}_{J\in A}$, or, in other words, by the action of the functional $L$ on the space  $\varinjlim \hat E^0_J$ of cylindrical functions.
\end{example}



\section{Continuous extensions of projective systems of functionals}

In the following we shall always assume, unless otherwise stated, that the projective family $(E_J,\pi_J^K)_{J,K\in A}$ is a perfect inverse system, in such a way that, by Theorem \ref{th-minimalextension}, the minimal extension $(L_{min},D(L_{min}))$ of a projective system of functionals $\{L_J,\hat E^0_J\}_{J\in A}$ is well defined. We shall study its possible continuous extensions.

Let us consider the set of cylindrical functions $\Cil_0$ associated to a projective system of functionals $\{L_J,\hat E^0_J\}_{J\in A}$ and let us endowe the space $\hat E_A$ (of complex valued functions on the projective limit $E_A$) with a topology $\tau$ in such a way that $(\hat E_A, \tau)$ becomes a topological vector space \footnote{We require a topological vector space in order to assure later on continuity under limits of sums and products by scalars.}. We say that the minimal extension is {\em closable} in $\tau$ if the closure of the graph $\cG (L_{min})$ of $L_{min}$
$$\cG (L_{min}):=\{ (f,L_{min}(f))\in \hat E_A\times \C : f\in D(L_{min}) \}$$
 in $\hat E_A\times \C$ with respect to the product topology $\tau\times \tau_\C$, $\tau_\C$ denoting the standard topology in $\C$, is the graph of a well defined functional. In this case we define the closure of $L_{min}$ in $\tau$ as the functional $\bar L_\tau$ such that its graph satisfies $\cG (\bar L_\tau)=\overline{\cG (L_{min})}$.\\
 If $(\hat E_A,\tau)$ satisfies the first axiom of countability, i.e. if any $f\in \hat E_A$ has a  countable neighbourhood basis, then the closability condition is equivalent to the requirement that for any sequence $f_n\in \Cil_0$ converging to $0$ and such that $L_{min}(f_n)\to z$, $z\in \C$, it follows that $z=0$.\\
In this case  the closure $(\bar L_\tau, D(\bar L_\tau))$ of $(L_{min},\Cil_0)$ in the $\tau$ topology is given by:
\begin{eqnarray*}
D(\bar L_\tau)&:=&\{  f\in \hat E_A\, |\, \hbox{ there exist } f_n\in\Cil_0, f_n\to f, L_{min}f_n\to z \}\\
\bar L_\tau (f)&:=&\lim_nL_{min}(f_n)=z, \quad f\in D(L_\tau).
\end{eqnarray*}
 $\bar L_\tau$ is well defined, indeed, by the closability condition, since the value of   $\bar L_\tau (f)$ for $f\in D(\bar L_\tau )$ does not depend on the sequence $\{f_n\}\subset \Cil_0$ converging to $f$.\\
Let us now assume that $L_{min}:\Cil_0\to \C$ is continuous \footnote{We recall that in a topological vector space if a functional is continuous in one point, then it is continuous everywhere} in the $\tau$ topology on $\Cil_0$, i.e. if for any $f\in \Cil_0$ and for any $\epsilon >0$ there exists a neighbourhood $U$ of $f$ (depending of $f$ and $\epsilon$) such that for any $g\in U\setminus \{f\}$, with $g\in \Cil_0$, one has $|L_{min} (f)-L_{min}(g)|<\epsilon$. Then it is easy to verify that $L_{min}$ is closable, as stated in the following theorem.

\begin{theorem}
Let $L_{min}:\Cil_0\to \C$ be continuous in the $\tau$ topology on $\Cil_0$. Then $L_{min}$ is closable.
\end{theorem} 

\begin{proof}
Let us consider two point $(g, \alpha), (g,\beta)\in \overline{\cG (L_{min})}$ and let us assume {\em per adsurdum } that  $\alpha\neq \beta$, and let us set  
 $\epsilon \equiv \frac{|\alpha-\beta|}{2}$. In this case  for any neighborhood $U(0)$ of $0$, it would be possible to find an element $\tilde f\in U(0)\cap \Cil_0$ such that $L(\tilde f)>\epsilon$, thus contradicting the assumed  continuity of the functional $L_{min}$.
By the properties of the topological vector spaces, given the  neighborhood $U(0)$ of $0$, then it is possible to find two  neighborhoods $V(0) ,W(0)$ of $0$ such that if $f_1\in V(0)$ and $f_2\in W(0)$, then $f_1-f_2\in U(0)$. Let $V(g):=V(0)+g$ and $W(g):=W(0)+g$ the neighborhoods of $g$ obtained by translating $V(0)$ and $W(0)$ by $g$. Since $(g, \alpha), (g,\beta)\in \overline{\cG (L_{min})}$, there is an $f\in V(g)\cap \Cil_0$ and an $f'\in W(g)\cap \Cil_0$ such that $|L(f)-\alpha|<\epsilon$ and $|L(f')-\beta|<\epsilon$. Then we have that $(f-g)\in V(0)$, $f'-g\in W(0)$ and $(f-g)-(f'-g)=f-f'\in U(0)$. Moreover $L(f-f')>\epsilon$, and $L$ cannot be continuous in $0$, against the assumption.
\end{proof}

Let us consider now the set $\Cil_0$ endowed with two different topologies $\tau _1$ and $\tau_2$, with $\tau_1$ coarser than $\tau_2$, i.e. $\tau_1\subseteq \tau_2$. 
If $(L_{min },D(L_{min}))$ is continuous in the $\tau_1$-topology, then it is continuous also in the $\tau_2$ topology and the closability in $\tau_1$ implies the closability in $\tau_2$. 
More generally, if $\tau_1$ coarser than $\tau_2$, then $\overline{\cG (L_{min})^{\tau_2}}\subset \overline{\cG (L_{min})^{\tau_1}}$ and if $\overline{\cG (L_{min})^{\tau_1}}$ is the graph of a well defined functional, then $\bar L_{\tau_1}$ and $\bar L_{\tau_2}$ are both well defined and $\bar L_{\tau_1}$ is a extension of $\bar L_{\tau_2}$.

Conversely, if $\tau_1$ and $\tau_2$ are not comparable, then there might exists a $f\in D(L_1)\subset D(L_2)$ such that $L_1(f)\neq L_2(f)$. This is the case of example \ref{es-not-uniqueness}. Indeed in this case the functional $L$ defined by \eqref{L-1} is continuous on $L^1(\R^\N, \mu_A)$, while the second functional \eqref{L-2} is continuous on the space $D(L')$ endowed with the norm
$$\Vert f\Vert_{L'}:=\lim_{n\to \infty}\int |f(\pi_n\omega)|\mu_A(d\omega) .$$
The set $\Cil_0$ of cylindrical functions is dense both in $D(L)$ and $D(L')$ with the respective topologies. On the other hand if we consider the sequence of cylindrical functions $f_n:=\Ep_{\R^n}^{\R^\N}{\bf 1}_{\R^n}$, where ${\bf 1}_{\R^n}:\R^n\to \R$ denotes the function identically equal to 1 in $\R^n$, one as that $f_n$ converges to the function $f=\chi_B$ ($B$ being as in example \ref{es-not-uniqueness}) in the $\Vert \; \cdot \; \Vert_{L'}$-norm, while it does not converge to $f$ in the $L^1(\R^N,\mu^A)$-norm.
One can easily see that these two norms are not comparable, as if we consider the function $f=\chi_B$, then $\|f\|_{L^1(\R^\N,\mu_A)}=0$ and $\|f\|_{L'}=1$, while if we consider the function $g=\chi_{\R^\N}-\chi_B$, then $\|g\|_{L^1(\R^\N,\mu_A)}=1$ and $\|g\|_{L'}=0$.


\section{Projective systems of complex measure spaces}

   Let us consider now the particular case where each element $(E_J)_{J\in A}$ of a projective family $\{E_J,\pi^K_J\}$ is endowed with a  $\sigma$-algebra $\Sigma_J$ of subsets of $E_J$. 
Let us also assume that the maps $\pi^K_J$, for $J\leq K$, are measurable. The family $ \{E_J,\Sigma_J, \pi^K_J\}$  is called a {\em projective family of measurable spaces} \cite{Xia,Yam,Yeh}. \\
Associated to each $(E_J,\Sigma_J)$ we assume that there is a   measure $\mu_J$, not necessarily real or positive. 
In particular we shall focus  on the case where $\mu_J$ is a  signed or complex measure with finite total variation \cite{Rudin,Tho}.

\begin{example}
A particular case of this situation is given as in example \ref{example1}. In fact we take $A$ to be set $A:=\Fin$ of all finite non empty subsets $J$ of an index set $I$, and  $E_J=\Pi_{t\in J}E^0_t$, where each $E^0_t$ is equipped with a $\sigma$-algebra  $\Sigma (E^0_t)$.  $E_J$ is thus a product space (cf. example \ref{example1}). In this case we can take $\Sigma_J$  to be the product $\sigma$-algebra $\otimes_{t\in J}\Sigma (E^0_t)$. We shall call $(E_J,\Sigma _J)$ a product measurable space. $\pi^K_J$ is then by construction a measurable map from $E_K $ to $E_J$, $K,J\in \Fin $, $J\leq K$.\\
The product measurable space $E:=\Pi_{t\in I}E^0_t$  is endowed with the smallest $\sigma$-algebra $\Sigma$ with respect to  which all projections $\pi^t:E\to E^0_t$ are measurable, i.e.
$$\Sigma:=\Sigma(\cup_{t\in I}(\pi^t)^{-1}(A_t)), \qquad A_t\in \Sigma(E^0_t).$$
$\Sigma$ is often denoted by $\otimes_{t\in I}\Sigma (E^0_t)$
\end{example}

Let us consider $L^1(E_J,\Sigma_J, \mu_J)$, the subset of $\hat E_J$ consisting of (real resp. complex) functions on $E_J$ which are $\mu_J$-integrable.

Let us consider the family of functionals $\{L_J,\hat E_J^0\}_{J\in A}$, given by
\begin{equation}\label{L-Prob}
\hat E_J^0:=L^1(E_J,\Sigma_J,\mu_J), \qquad L_J(f):=\int_{E_J} fd\mu_J, \, f\in \hat E_J^0.
\end{equation}

 In this case  $L_J$ is a linear functional (real resp. complex valued).
$\Ep^K_J$ is defined as before, as a map from  $\hat E_J$ to $\hat E_K$, $J\leq K$, $J,K\in A$   \footnote{Note that in the general case if $f_J\in \hat E_J$ is measurable (w.r.t. $\Sigma_J$) then $\Ep_J^K(f_J)$ is measurable (w.r.t. $\Sigma_K$, $J\leq K$). However in general it is not true that $\Ep_J^K(\hat E_J^0)\subseteq \hat E_K^0$. The condition $\Ep_J^K(\hat E_J^0)\subseteq \hat E_K^0$ is automatically fulfilled in the case where $\hat E_J^0$ is the set of all bounded measurable functions on $(E_J,\Sigma_J)$, instead of $L^1(E_J,\Sigma_J,\mu_J)$, or we have the situation of a projective family, as in \eqref{eq5} and \eqref{eq6}}. 

 According to definition  \ref{def3} the family $\{L_J,\hat E_J^0\}_{J\in A}$ is projective on $\hat E_J^0$ if for all $J\leq K$, $K,J\in A $, $f_J\in \hat E^0_J$ the following compatibility conditions hold:
$$
\Ep_J^K(f_J)\in \hat E^0_K,$$ 
\begin{equation}\label{eq5}
\hat \pi^K_J(L_K)(f_J)=L_J(f_J).
\end{equation}

Due to the relation between  $L_J$ and $\mu_J$ and \eqref{eq3}, we have that \eqref{eq5} implies 
\begin{equation}\label{eq6}
\int_{E_J} f_Jd\mu_J=\int_{E_K} f_J\circ \pi^K_Jd\mu_K, \qquad f_J\in L^1(E_J,\Sigma _J,\mu_J).
\end{equation}
For $f_J$ taken  to be the characteristic function of $A_J\in \Sigma _J$ this implies
\begin{equation}\label{eq7}
 \pi^K_J(\mu_K)=\mu_J,
\end{equation}
which is the analogue of the usual projectivity property for measures on projective  spaces (see, e.g., \cite{Bau}). We shall say shortly that $\{\mu_J\}_{J\in A}$ is a projective family of  measures.
Conversly, if \eqref{eq7} holds, then by approximating $L^1$-functions by finite linear combinations of characteristic functions we have that \eqref{eq6} holds, which by the relation \eqref{L-Prob} between $L_J$ and $\mu_J$ implies that \eqref{eq5} holds. Hence the family $(L_J,\hat E^0_J)_{J\in A}$ is projective   iff $\mu_J$ is a projective family of  measures (in the sense of \eqref{eq7}). 

Given a projective family $(E_J, \Sigma_J,\pi_J^K)_{J,K\in A}$ of measure spaces its projective limit $(E_A, \Sigma_A)$ is the measure space defined as $E_A=\varprojlim E_J$ and $\Sigma_A:=\Sigma_\infty\cap E_A$, where $\Sigma_\infty=\otimes_{J\in A}\Sigma_J$ is the $\sigma$- algebra associated with the product space $\Pi_{J\in A}E_J$, which coincides with the smallest $\sigma$-algebra making all projection maps  $\tilde\pi_J$ from
$\Pi_{J\in A}E_J$ onto $E_J$ measurable. By construction we thus have that $\Sigma_A=\cup_{J\in A}\pi_J^{-1}\Sigma_J$ is the $\sigma $-algebra generated by the sets of the form $\pi_J^{-1}(B)$, $B\in \Sigma_J$ ($\pi_J$ being defined in section 2, before Eq. \ref{rel-pi}).
 We shall write 
$$(E_A,\Sigma_A)=\varprojlim (E_J, \Sigma_J).$$
Let $\mu $ be a (complex bounded) measure on $(E_A,\Ba(E_A))$, and let us define the measures $\mu_J$ on $(E_J,\Sigma_J)$  by:
\begin{equation}\label{proj-measure}
\mu_J:=\pi_J\circ \mu, \qquad J\in A.
\end{equation} 
 It is easy to verify that the family of measures $(\mu_J)_{J\in A}$ satisfies the compatibility condition
\begin{equation}\label{comp-cond-mu}
\pi_J^K\circ \mu_K=\mu_J,\qquad J\leq K
\end{equation}
More generally, we shall say that the members of family of measures $(\mu_J)_{J\in A}$ on $(E_J,\Sigma_J)_{J\in A}$ are {\em compatible} (or shortly {\em the family is compatible}) if they satisfy the compatibility condition \eqref{comp-cond-mu}. 

Let us consider now the converse problem to the one solved by \eqref{proj-measure}, namely about when we can find $\mu$ on $(E_A,\Sigma_A)$ such that \eqref{proj-measure} holds, starting only from integration on $(E_J,\Sigma_J)_{J\in A}$ and condition \eqref{comp-cond-mu}.  
If such a signed or complex measure $\mu$ exists, then the projective system of functionals  $\{L_J,\hat E_J^0\}_{J\in A}$ given by \eqref{L-Prob} would admit a projective extension $L$ given by 
$$D(L):=L^1(E_A,\Sigma_A,\mu)$$
\begin{equation}\label{L-Prob-2}
L(f):=\int_{E_A} fd\mu, \qquad f\in D(L).
\end{equation}

If $\{\mu_J\}_{J\in A}$ is a projective system of compatible {\em probability} measures, then Kolmogorov's existence theorem \cite{Boc,Bau,Xia,Yam,Yeh,HofJor} assures the existence and uniqueness of $\mu$.
On the other hand if the measures $\{\mu_J\}_{J\in A}$ of the projective family are not real and positive, i.e. if we consider a projective family of {\em signed} or {\em complex} bounded measures, then in general   such a $\mu$ cannot exists, as stated in the following theorem.

\begin{theorem}\label{th-b-var}
Let $(E_J,\Sigma_J,\pi_J^K)_{J,K\in A}$ be a projective family of measure spaces and let  $\{\mu_J\}_{J\in A}$ be a projective family of signed or complex bounded measures, satisfying the compatibility condition  \eqref{comp-cond-mu}. A necessary condition for the existence of a  (signed or complex) bounded measure $\mu$ on $(E_A,\Sigma_A)$ satisfying the relation \eqref{proj-measure} is the following uniform bound on the total variation of the measures belonging to the family $\{\mu_J\}_{J\in A}$:
\begin{equation}\label{un-bound}
\sup_{J\in A}|\mu_J|<+\infty
\end{equation}
where $|\mu_J|$ denotes the  total variation of the measure $\mu_J$.
\end{theorem} 

\begin{proof}
Let $\{A_i\}\subset \Sigma_J$ be a partition of $E_J$ consisting of sets $A_i$ in $\Sigma_J$. Then $\{\pi_J^{-1}(A_i)\}\subset \Sigma_A$ is a partition of $E_A$. If a signed or complex measure $\mu$ on $(E_A,\Sigma _A) $ satysfying \eqref{proj-measure} exists, then for any $J\in A$ its total variation  would satisfy the 
 following inequality $|\mu|\geq |\mu_J|$. Consequently, if the signed or complex measure is bounded, then
$$\sup_{J\in A}|\mu_J|\leq |\mu|<+\infty.$$
\end{proof}

In the case of a projective family of probability measures the condition \eqref{un-bound} is trivially satisfied. In the general case of complex measures satisfying  \eqref{un-bound} sufficient conditions for the existence of the measure $\mu $ on $(E_A,\Sigma_A)$ have been given in  \cite{Tho,Bog}, where Kolmogorov-Prokhorov's theorem for the projective limit of probability measures \cite{HofJor,Bou,Schwartz,Neveu} has been generalized to the case of signed or complex measures.

\begin{theorem}\label{teo5}
Let $(E_J,\pi_J^K)_{J,K\in A}$ be a projective system of Hausdorff topological spaces, where the map $ \pi_J^K:E_K\to E_J$, $J,K\in A$, $J\leq K$ are continuous. Let $(\mu_J)_{J\in A}$ be a family of Radon measures (in the sense of \cite{Bou,Tho}) satisfying the projectivity condition \eqref{comp-cond-mu}. 
Further assume that the spaces $E_J$ and the space $E_A=\varprojlim E_J$ are completely regular. Then there exists a (signed resp. complex bounded) Radon measure $\mu$ on $E_A$ satisfaying \eqref{proj-measure} if and only if the following two conditions are satisfied:
\begin{enumerate}
\item[i.] $\sup_J |\mu_J|<+\infty$, where $\mu_J$ denotes the total variation of the measure $\mu_J$,
\item[ii.] for every $\epsilon >0 $ there exists a compact $K\subset E_A$ such that if $K_J=\pi_J(K)$, the following holds
$$|\mu_J|(E_J\setminus K_J)<\epsilon, \qquad \forall J\in A$$
\end{enumerate}
Under these assumptions $\mu$ is uniquely determined and $|\mu|=\sup_{J}|\mu_J|$.
\end{theorem}

For a complete proof see \cite{Tho}.\\
 On the other hand, if $\mu_J$ are signed or complex bounded measures, in many interesting cases condition \eqref{un-bound} cannot be satisfied, as illustrated in  the next examples in the following subsections \ref{sez4.2} and \ref{sez4.3}.\\
If there does not exist a measure $\mu $ on $(E_A,\Sigma_A)$ obtained as the projective limit of the measures $\mu_J$, there cannot exist a projective extension $(L,D(L))$ of the projective system of functional $\{L_J,\hat E^0_J\}_{J\in A}$ of the form \eqref{L-Prob} (i.e. associated to an integral). On the other hand it is possible to define alternative projective extensions, as we shall discuss in section \ref{sez6}.


\subsection{The special case of product measures}\label{sez4.1}

Let $A\equiv \Fo(\N)$ be the direct set  of finite subsets of $\N$ and let $\{(E_n, \Sigma_n, \mu_n)\}_{n\in\N}$  be a sequence of measure spaces. Let us assume that for any $n\in \N$, the measure $\mu_n$ is a complex bounded measure such that 
$\int_{E_n}d\mu_n=1$. For any $J\in A$, let $E_J$ be the product space $E_J\equiv \Pi_{n\in J}E_n$, endowed with the product $\sigma $-algebra $\Sigma_J\equiv \otimes_{n\in J}\Sigma_n$ and the product measure $\mu_J=\times_{n\in J}\mu_n$. Then one can easily verify that $(E_J, \Sigma_J,\mu_J)_{J\in A}$ is a projective system of measure spaces.
The projective limit measure space $(E_A,\Sigma_A)$ is naturally isomorphic to the product space $(\times_{n\in \N}E_n,\otimes_{n\in \N}\Sigma_n )$ and in the following we shall write $E_A\equiv  \times_{n}E_n$ and $\Sigma_A\equiv \otimes_{n}\Sigma_n $.    
  Further, if a signed or complex bounded variation measure $\mu$ which is the projective limit of the projective system of product measures $(\mu_J)_{J\in A}$ exists, then for any $N\in \N$ its total variation would be underestimated by $\Pi_{n=1}^N \|\mu_n\|$. Thus  a necessary condition for the existence of the projective limit measure is 
\begin{equation}\label{prod-meas}
\Pi_n \|\mu_n\|<+\infty
\end{equation}
Equivalently one has to require the convergence of  the series $\sum_n \log (\|\mu_n\|)$.\\
This condition cannot be satisfied, for instance, in the case where for any $n\in\N$, one has that $(E_n,\Sigma_n,\mu_n)=(E,\sigma,\mu)$ and $ \mu $ is a signed or complex bounded measure such that $\int_E d\mu =1$ and $\|\mu\|=c>1$.\\

Analogously, in the case of a continuous product space, i.e. if  $A\equiv \Fo(I)$, $I\subset \R$ is an interval of the real line, then the necessary condition for the existence of a signed or complex bounded  product measure is
\begin{equation}\label{prod-meas2}
\sup_{J\in A} \{\sum_{t\in J}\log \|\mu_t\| \}<+\infty.
\end{equation}


\subsection{Semigroups of complex kernels and pseudoprocesses}\label{sez4.2}

Let $(E,\Sigma)$ be a measurable space. A {\em complex kernel} $K$ on $(E,\Sigma)$  is a map $K:E\times \Sigma\to \C$ with the following properties
\begin{itemize}
\item[i.] the map $x \in E\mapsto K(x,B)$ is measurable for each $B\in \Sigma$,
\item[ii.] the map $B\in \Sigma \mapsto K(x,B)$ is a complex bounded variation measure on $\Sigma$ for each $x\in E$
\end{itemize}

This concept generalizes the one of {\em Markov (resp. sub-Markov) kernel}, where $K:E\times \Sigma\to \R_+$, $K(x,E)=1$ (resp. $K(x,E)\leq 1$) for every $x \in E$, and in the condition ii. ``complex'' is replaced by ``positive''. For Markov, resp. sub-Markov, kernels see e.g. \cite{Bau}.\\

Given two complex kernels $K_1,K_2$ on  $(E,\Sigma)$, their composition $K_1\circ K_2$ is the complex kernel defined as 
$$K _1\circ K_2(x,B):=\int_E K_1(x,dx')K_2(x', B), \qquad x\in E, B\in \Sigma.$$
Let $\{K_t\}_{t\geq 0}$ be a family of complex kernels on a measurable space $(E,\Sigma)$ with index set $\R_+$ such that
$K_{t+s}=K_t\circ K_s$ for all $t,s\in \R_+$, namely
\begin{equation}\label{C-K}K_{t+s}(x,B)=\int_E K_t(x,dx')K_s(x', B) ,\qquad x\in E, B\in \Sigma.\end{equation}
The family $\{K_t\}_{t\geq 0}$ will be called {\em semigroup of complex kernels}. 
Given a semigroup of complex kernels, it is possible to construct a projective family of complex measures. \\ 
Let $A=\Fo(\R_+)$ be the directed set of finite subsets of $\R_+$  and for any $J \in A$ let   $E_J:=\Pi_{t\in J}E$ endowed with the product $\sigma$ -algebra $\Sigma_J:=\otimes _{t\in J}\Sigma$. Let   $\nu$ be a complex bounded measure on $(E,\Sigma)$. For $J=\{t_1,t_2,...,t_n\}$, with $0<t_1<t_2<...<t_n<+\infty$,  let $\mu_J$ be the complex measure on $(E_J,\Sigma_J)$ given by 
\begin{multline}\label{mu-K}
\mu_J(B_1\times B_2\times...\times B_n)=\int_E\int_{B_1}...\int_{B_n}K_{t_n-t_{n-1}}(x_{n-1},dx_n)\dots \\ ...K_{t_2-t_1}(x_1,dx_2) K_{t_1}(x_0,dx_1)d\nu(x_0)
\end{multline}
By the semigroup property \eqref{C-K}, one can easily verify that $(\mu_J)_{J\in A}$ forms a projective system of complex measures, i.e. for any if $J,K\in A$, with $J\leq K$, one has $\mu_J=\pi_J^K\circ \mu_K$. In fact it is sufficient to verify this conditions for $J,K\in A$ such that $K\setminus J$ consists of exactly one element, since the case where $K\setminus J$ consists of $m>1$ points can be studied by constructing a chain $J\subset K_1\subset K_2\subset ...\subset K_m=K$ where each difference $K_i\setminus K_{i-1}$ is a singleton, and using the identity
$$\pi_J^K=\pi_{K_{m-1}}^{K}\circ \pi^{K_{m-1}}_{K_{m-2}}\circ... \circ \pi_{J}^{K_1}.$$
Let $J=\{t_1,t_2, ..., t_n\}$ and $K\setminus J=t'$, with $t_i<t'<t_{i+1}$. We have to show that for any $B\in \Sigma_J$, $\mu_J(B)=\pi^K_J\circ \mu_K(B)$. It is sufficient to show this equality for $B\in \Sigma_J$ of the form $B=B_1\times B_2\times...\times B_n$, with $B_i\in \Sigma$ for all $i=1,..,,n$, that is
\begin{multline}\label{eq19}
\int_E\int_{B_1}..\int_{B_i}\int_{B_{i+1}}...\int_{B_n}K_{t_n-t_{n-1}}(x_{n-1},dx_n)\dots K_{t_i-t_{i-1}}(x_{i-1},dx_i)\\ K_{t_{i-1}-t_{i-2}}(x_{i-2},dx_{i-1})... K_{t_1}(x_0,dx_1)d\nu(x_0)\\
=\int_E\int_{B_1}..\int_{B_i}\int_E\int_{B_{i+1}}...\int_{B_n}K_{t_n-t_{n-1}}(x_n-1,dx_n)\dots K_{t_i-t'}(x',dx_i)\\K_{t'-t_{i-1}}(x_{i-1},dx')K_{t_{i-1}-t_{i-2}}(x_{i-2},dx_{i-1})... K_{t_1}(x_0,dx_1)d\nu(x_0),
\end{multline}
which holds by the semigroup property \eqref{C-K}.

According to Theorem \ref{th-b-var}, a necessary condition for the existence of a complex measure $\mu$ on $(E^{\R^+}, \Sigma)$ is the uniform bound \eqref{un-bound}. If $K_t(x, \, \cdot\,)$ is a probability measure for any $x\in E$ and $t\in \R^+$, then \eqref{un-bound} is trivially satisfied. In this case, under further assumptions, for instance if $(E,\Sigma)$ is a Polish space (see e.g. \cite{Bau} for further details) there exist a probability (resp. sub-probability) measure $\mu$ on $(E^{\R^+}, \Sigma)$ associated to the projective family $(\mu_J)_{J\in A}$ of compatible probability measures. It describes a stochastic (in fact Markov, resp. sub-Markov) process $X_t$, $t\geq 0$. The finite dimensional distributions $P(X_{t_1}\in B_1, ..., X_{t_n}\in B_n)$, $t_1\leq ...\leq t_n$, $B_i\in \Sigma)$ for $i=1,...,n$, are given by 
\begin{eqnarray}\label{fin-dim-dist}
& &\mu(\pi_J^{-1}(B_1\times ...\times B_n))=\\
& &P(X_{t_1}\in B_1, ..., X_{t_n}\in B_n)=\mu_J(B_1\times ...\times B_n), \quad J=\{t_1,..., t_n\}
\end{eqnarray} (see. e-g.,  \cite{Bau}). Let us also remark that if the index set $\R_+$ for the times is replaced by a countable directed set $T$ the result holds without any topological assumption on $(E,\Sigma)$ (with $E^{\R^+}$ replaced by $E^T$). This is a theorem of C. Ionescu Tulcea, see, e.g. \cite{Neveu}.\\
 In the case of complex or signed kernels $K$ under the assumptions of Theorem \ref{teo5} the measure $\mu$ given in Theorem \ref{teo5} exists and one can also introduce the concept of {\em pseudoprocesses} $X_t$, $t\geq 0$, with the family of ``finite dimensional distributions'' given by the $\mu_J$ in the sense that $X_t(\omega)=\pi_t\omega$, $\omega\in E^{\R_+}$ and  \eqref{fin-dim-dist} holds again (without of course a probabilistic interpretation). \\
On  the other hand, if $K_t(x, \, \cdot\,)$, with $x\in E$ and $t\in \R^+$, are general complex or signed measures, in many interesting cases  condition \eqref{un-bound} is not satisfied.\\
Let us consider for instance the case where 
$E$ is a vector space and the translation is a measurable map, i.e. for any $x\in E$ and $B\in \Sigma$ the set $B+x:=\{x'\in E\, |:\, x'-x\in B\}$ is measurable. 
A semigroup of complex kernels $K_t$ on $(E,\Sigma)$ is said to be {\em translation invariant} if $$K_t(x,B)= K_t(x+x',B+x'), \qquad  \forall x,x'\in E, B\in \Sigma.$$ In this case, by setting $\mu_t(B):=K_t(0,B)$, the semigroup law \eqref{C-K} reads as $\mu_{t+s}=\mu_t*\mu_s$ (the symbol $*$ denoting the convolution of measures) and one gets a convolution semigroup of (complex bounded) measure on $(E,\Sigma)$.

If $\mu _t$ are  probability measures on $E=\R^d$ (or on a real separable Hilbert space) then the above probability measure $\mu$ on $(E^{\R^+},\Sigma)$ describes a Markov process with stationary independent increments \cite{Bau,Partha}. In the case where $\mu_t$ are complex valued, if $J=\{t_1,t_2,...,t_n\}$, with $0<t_1<t_2<...<t_n$, the total variation of the measure $\mu_J$ is given by:
$$|\mu_J|= |\nu|\, |\mu_{t_1}||\mu_{t_2-t_1}|\cdots |\mu_{t_n-t_{n-1}}|.$$ 
(with $\nu$ as in \eqref{eq19}).
In particular, if $t_j\equiv \frac{j}{n}$, with $j=1,...,n$, we have
$|\mu_J|= |\nu|\,( |\mu_{1/n}|)^n$ and the behaviour of the total variation of $\mu_J$ for $n\to \infty$ is steered by the small time asymptotics of $|\mu_t|$ for $t\to 0$. 

A particular example is provided by considering a projective family of compatible signed measures on the algebra of cylindrical sets in $\R^{[0,+\infty)}$ constructed in terms of the fundamental solution $G_t$ of the PDE:
\begin{eqnarray*}
\frac{\partial}{\partial t} u(t,x)&=&(-1)^{N+1}(\Delta _x)^N u(t,x), \qquad x\in \R, t\in \R^+,\\
u(0,x)&=&f(x),
\end{eqnarray*}
with $f\in S(\R)$ (e.g.). We have
 $u(t,x)=\int_\R G_t(x-y)f(y)dy$, where 
\begin{equation}\label{Green}
G_t(x-y)=\frac{1}{2\pi}\int_\R e^{ik(x-y)}e^{-t k^{2N}}dk,
\end{equation}
where $N\in \N$, $N\geq 2$. 
By the Fourier integral representation \eqref{Green}, it follows that for $t>0$ the distribution $G_t\in {\mathcal S}'(\R)$ is a $C^\infty(\R)$ function, whereas $G_0()$ is Dirac's measure concentrated at $0$. Moreover an analysis of the asymptotic behaviour of $G_t(x-y)$ as $|x-y|\to +\infty$ shows that it belongs to $L^1(\R) $ and, for $N\geq 2$ it is not positive but has an oscillatory behaviour, changing sign an infinite number of times \cite{Hoc,Ma14}. In the case of $N=1$ one simply has the heat kernel, and $G_t>0$.
 
By considering the signed measure $\mu_t$ on $\R$ of the form $d\mu_t(x)=G_t(x)dx$, $dx$ denoting the Lebesgue measure on $\R$, and constructing the family of complex measures $(\mu_J)_{J\in A}$ of the form \eqref{mu-K}, with $K_t(x,B):=\int_B G_t(x-y)dy$ and $\nu=\delta_0$, by the semigroup property 
$$\int_\R G_t(x-z)G_s(z-y)dz=G_{t+s}(x-y)$$
one obtains a projective family of signed measures. By considering now $J=\{t_1,t_2,...,t_n\}$, with $t_j\equiv \frac{j}{n}$,  $j=1,...,n$, we see that the total variation of the measure $\mu_J$ is given by 
$$|\mu_J|=\left(\int_\R |G_{1/n}(x)|dx\right)^n.$$
On the other hand, by Eq. \eqref{Green}:
$$\int _\R|G_{1/n}(x)|dx =\int _\R n^{1/2N}|G_1(n^{1/2N}x)|dx=\int _\R |G_1(y)|dy,$$
and since for $N>1$ the smooth  function $G$ has an oscillatory behaviour and   $\int _\R G_1(y)dy=1$, one concludes that for any $n\in \N$ $$\int _\R|G_{1/n}(x)|dx=\int _\R |G_1(y)|dy=c>1,$$
where $c$ is a constant depending on $N$,
and $|\mu_J|=c^n\to +\infty $ as $n\to + \infty$ (for $N=1$ there is no such result since $G_t\geq 0$ and $c$ would be equal to 1) By theorem \ref{th-b-var} there does not exists a signed bounded measure $\mu$ on $\R^{[0,+\infty)}$ defined on the algebra of   cylindrical sets
$I_k\subset\Omega=\{x:[0,\infty)\to\R\}$ of the form  $$I_k:=\{\omega\in \Omega: \omega (t_j)\in [a_j,b_j], j=1,\dots k\},\quad  0<t_1<t_2<\dots t_k,$$ with
\begin{equation}\label{CylMeas}
\mu(I_k)=\int_{a_1}^{b_1}...\int_{a_k}^{b_k}\prod_{j=0}^{k-1}G_{t_{j+1}-t_j}(x_{j+1}-x_j)dx_1...dx_{k}.
\end{equation} 
Hence the projective system of functionals $(L_J, \hat E^0_J)$, defined by
\begin{eqnarray*}
\hat E^0_J&=&L^1(E_J,|\mu_J|), \qquad E_J=\R^{J},\\
L_J(f)&=&\int _{E_J}fd\mu_J
\end{eqnarray*}
cannot have a projective extension $(L,D(L))$ of the form
\begin{eqnarray*}
D(L) &=&L^1(E_A,|\mu|), \qquad E_A=\R^{[0,+\infty)},\\
L(f)&=&\int _{E_A}fd\mu.
\end{eqnarray*}
For a detailed analysis of this problem see, e.g., \cite{Kry,Hoc,Ma13,BoMa14,Ma14,FaS}.


\subsection{Fresnel and Feynman path integrals}\label{sez4.3}

Let us consider on $\R^n$ the complex measure $\mu_n$ absolutely continuous with respect to the Lebesgue measure $dx$, $x\in\R^n$, with a density of the form $ \rho (x)=\frac{e^{\frac{i}{2}\|x\|^2}}{(2\pi i)^{n/2}} $. The total variation of $\mu_n$ on $\R^n$ is infinite, however for any Borel bounded set $B\subset \R^n$, the total variation of $\mu_n$ on $B$ is finite. Given a bounded Borel function $f:\R^n\to \C$ the integral of $f$ on $B$ with respect to $\mu_n$: $$\int_Bf(x)\frac{e^{\frac{i}{2}\|x\|^2}}{(2\pi i)^{n/2}}dx$$ is well defined.
More generally, the integral of a complex bounded Borel function on $\R^n$ with respect to $\mu_n$ can be defined as the limit
$$\lim_{R\to\infty }\int_{[-R,R]^n}f(x)\frac{e^{\frac{i}{2}\|x\|^2}}{(2\pi i)^{n/2}}dx$$
if this limit exists. In this case it is denoted by $\tilde \int _{\R^n}f(x)\mu_n(dx)$. According to this definition and the properties of the classical Fresnel integrals one get that the $\tilde \int$-integral of the function identically equal to 1 on $\R^n$ is 1, i.e. $\tilde \int _{\R^n}d\mu_n=1$.\\
The complex measure $\mu_n$ can be represented as the product of $n$ copies of the complex measure $\mu$ on $\R$, where $d\mu:=\frac{e^{\frac{i}{2}\|x\|^2}}{\sqrt{2\pi i}}dx$. \\
Let $A=\Fo(\N)$ be the directed set of finite subsets of $\N$ and let us consider for any $J\in A$ the set $E_J:=\R^J$ endowed with the Borel $\sigma-$algebra and the complex measure (with finite total variation on bounded sets) $\times _{n\in J}d\mu$.  Let us define for any $J\in A$ the functional $L_J:\hat E^0_J\to\C$, given by:
\begin{eqnarray*}
& &\hat E^0_J:=\{ f\in \Ba_b(E_J)\, :\, \exists \lim _{R\to +\infty}\int _{[-R,R]^{|J|}}f(x)\times _{n\in J}\mu(dx)\} \\
& &L_J(f):=\lim _{R\to +\infty}\int _{[-R,R]^{|J|}}f(x)\times _{n\in J}\mu(dx), \quad f\in \hat E^0_J.
\end{eqnarray*}
One can easily verify that $(L_J,\hat E^0_J)_{J\in A}$ is a projective system of functionals. However it is impossible to construct a projective extension on $E_A\equiv \R^N$ in terms of a (Lebesgue type) improper integral. 
Indeed, contrary to the case of finite dimension, if we consider the infinite product measure $\times_{n\in\N}d\mu$, on $\R^\N$ endowed with the product $\sigma$ algebra, we have that its total variation is infinite even on products of bounded sets.

More generally (as in example \ref{example2}) let $(\Hi,\langle \; ,\; \rangle)$ be a real separable Hilbert space and let $A$ be the directed set of is finite dimensional subspaces, where $V\leq W$ if $V$ is a subspace of $W$, and $\pi^W_V:W\to V$ for $V\leq W$ is the natural projection from $W$ onto $V$. For any $V\in A$ let $\Sigma_V$ be the Borel $\sigma$-algebra on $V$. $(V,\Sigma_V, \pi_V^W)_{V,W\in A}$ is then  a projective system of measure spaces.
For any $V\in A$ let $L_V: D(L_V)\to \C$ be the linear functional defined by:
\begin{eqnarray*}
D(L_V)&:=&\{ f\in \Ba_b(E_J)\, :\, \exists \lim _{R\to +\infty}\int _{[-R,R]^{|V|}}f(x)e^{\frac{i}{2}\|x\|^2}dx\} \\
L_V(f)&:=&\lim _{R\to +\infty}\int _{[-R,R]^{|V|}}f(x)\frac{e^{\frac{i}{2}\|x\|^2}}{(2\pi i)^{|V|/2}}dx, \quad f\in \hat D(L_V).
\end{eqnarray*}
where $dx$ denotes the Lebesgue measure on $V$, $\| \; \|$ is the norm in $V$  and $|V|$ denotes the dimension of $V$.\\
The family $(L_V,D(L_V))_{V\in A} $  costitutes a projective system of linear functionals. As previously remarked it is not possible to define on the projective limit $E_A=\R^\N$ a complex measure with finite variation on bounded sets obtained as the projective limit of the complex measures $\mu_V(dx):=\frac{e^{\frac{i}{2}\|x\|^2}}{(2\pi i)^{|V|/2}}dx$, $x\in V$. Consequently there cannot be an extension of the projective system of functionals $(L_V.D(L_V)_{V\in A} $ of the form 
$L(f)=\int_{E_A}f(x)\mu(dx)$, even if $f$ is supported in a product of bounded sets.

A similar phenomenon occurs in the mathematical construction of Feynman path integrals.
Let us consider the fundamental solution $G_t\in {\mathcal S}'(\R^d)$ of the Schr\"odinger equation for a non-relativistic quantum particle moving freely in the $d-$dimensional Euclidean space:
\begin{eqnarray}
i\hbar \frac{\partial }{\partial t}\psi (t,x)&=&-\frac{\hbar^2}{2} \Delta \psi (t,x)\nonumber\\
\psi (0,x)&=&\psi_0(x)\qquad\qquad t\in\R, \, x\in \R^d
\end{eqnarray} 
($\hbar$ being the reduced Planck constant). We have then
 $$\psi(t,x)=\int_{\R^d}G_t(x-y)\psi_0(y)dy, \qquad \psi_0\in {\mathcal S}(\R^d),$$
where, for $t\neq0$, $G_t$ is the $C^\infty(\R^d)$ function 
\begin{equation}\label{Green-Sch}
G_t(x-y)=\frac{e^{\frac{i}{2t\hbar}|x-y|^2}}{(2\pi i t \hbar)^{d/2}}.
\end{equation}
$G_t$ can be regarded as the density with respect to the Lebesgue measure of a complex measure $\mu_t$ on $\R^d$, with finite variation on bounded sets, of the form $d\mu_t(x)=G_t(x)dx$.  $G_0$ is defined as the Dirac measure at $0$. Further, by the equation $\int_{\R^d}G_t(x-z)G_s(z-y)dz=G_{t+s}(x-y)$, $s,t\in\R$ (where the integral is meant as an improper Riemann integral) one has that the family of functionals $(L_J,\hat E^0_J)$ defined by :
\begin{eqnarray*}
\hat E^0_J&=&L^1(E_J,dx) ,\quad E_J=(\R^d)^J, J=\{t_1,...,t_n),\\
L_J(f)&\equiv& \int_{\R^J} f(x_1,...,x_n)G_{t_1}(x_0,x_1)...G_{t_n-t_{n-1}}(x_{n-1},x_n)dx_1...dx_n
\end{eqnarray*}
is a projective family of functionals, but there cannot exist a complex bounded variation measure $\mu$ on $(\R^d)^{[0,+\infty)}$ and a projective extension $(L,D(L))$ of  $(L_J,\hat E^0_J)$ such that
$L(f)=\int_{(\R^d)^{[0,+\infty)}}fd\mu$ (see \cite{Cam}). This problem is deeply connected with the rigorous mathematical definition of Feynman path integrals which has been provided in different ways, but always, because of the above obstruction, only in the sense of continuous functionals not directly expressible as integrals with respect to $\sigma$-additive measures \cite{AlHKMa,HKPS,JoLa,Ma,Ma11,Bor}. In section \ref{sez6} we shall provide some examples.
\section{Examples of projective extensions}\label{sez6}

In this section we present two examples of possible extensions of projective system of functionals of the form \eqref{L-Prob}, associated to a projective system of complex or signed measures, and show their application to the construction of Feynman-Kac type formule for the representation of the solution of PDEs which do not satisfy a maximum principle.\\

Let us consider the general setting of example \ref{example2}, i.e. let $\Hi$ be a real separable infinite dimensional Hilbert space and let $A$ be the directed set of its finite dimensional subspaces, ordered by inclusion. For $V,W\in A$, with $V\leq W$, let $\pi_V^W:W\to V$ be the projection from $W$ onto $V$ and $i_V^W:V\to W$ be the inclusion map. One has that $(V,\pi_V^W)_{V,W\in A}$ is a  projective family of sets, while $(V,i_V^W)_{V,W\in A}$ forms a direct system on $A$. Let us consider the projective limit space
$E_A:=\varprojlim_{V\in A} V$, the direct limit $\tilde E_A:=\varinjlim_{V\in A}$,  the projection  $\pi_V:E_A\to V$ and the inclusion maps $i_V:V\to \tilde E_A$. Considered on each $V\in A$ the topology induced by the finite dimensional Hilbert space structure of $V$, the space $\varprojlim_{V\in A} V$ is endowed with the weakest topology making all the projections $\pi_V:E_A\to V$ continuous, while the space $\varinjlim_{V\in A}$ is endowed with the {\em final topology}, i.e. the strongest topology making all the inclusion maps $i_V:V\to \tilde E_A$ continuous.\\
The inverse system $(V,\pi_V^W)_{V,W\in A}$ and direct system $(V,i_V^W)_{V,W\in A}$ are linked by dualization. Indeed if we identify the dual of a finite dimensional vector space $V$ with $V$ itself, we have that the inclusion $i_V^W:V\to W$, $V\leq W$, can be identified with the transpose map $(\pi_V^W)^*:V^*\to W^*$ of the projection $\pi_V^W:W\to V$.
Further the direct limit space $\tilde E _A$ can be naturally identified with a subspace of $(E_A)^*$. Indeed any $\eta\in \tilde E_A$ can be associated with the element of $(E_A)^*$ whose action on any $\omega\in E_A$ is given by
\begin{equation}\label{dual}
\eta(\omega):=\langle v, \pi_V\omega\rangle,
\end{equation} 
$v\in V$ being any representative of the equivalence class of vectors associated to $\eta$.
The definition \eqref{dual} is well posed, indeed chosen a different representative of the equivalence class $\eta$, i.e. a vector $v'\in V'$ such that there exists a $W\in A$, with $V\leq W$, $V'\leq W$ and $i_V^Wv=i_{V'}^Wv'$, one has that:
$$ \langle v, \pi_V\omega\rangle=\langle v, \pi_V^W\circ \pi_W\omega\rangle=\langle i_V^W v, \pi_W\omega\rangle=\langle i_{V'}^W v', \pi_W\omega\rangle=\langle v', \pi_{V'}\omega\rangle.$$
Further the explicit form \eqref{dual} of the functional $\eta$ shows its continuity on $E_A$.\\ 
Analogously the transpose map $\pi_V^*:V^*\to E_A^*$ can be identified with  the map $i_V:V\to \tilde E_A$, giving:
$$\langle v,\pi_V\omega\rangle=\,_{E_A^*}\langle i_V v,\omega\rangle_{E_A},$$
where the symbol $\langle \;,\;\rangle$ on the left hand side denotes the inner product in $V$, while symbol $\,_{E_A^*}\langle \;,\,\rangle_{E_A}$ denotes the dual pairing between $E_A$ and $E_A^*$.\\

Let us consider on any $V\in A$ the Borel $\sigma $-algebra  $\Sigma_V$ and a bounded signed or complex measure $\mu_V:\Sigma_V\to \C$ in such a way that the family $(\mu_V)_{V\in A}$ satisfies the compatibility condition \eqref{comp-cond-mu}. Let us also consider, for any $V\in A$, the Fourier transform  $\hat\mu_V:V\to \C$ of the measure $\mu_V$, i.e.
$$\hat \mu_V(v)=\int_Ve^{i\langle v',v\rangle}\mu_V(dv'), \qquad v\in V.$$
By the projectivity condition \eqref{comp-cond-mu} of the family of measures $(\mu_V)_{V\in A}$, one deduces the following relation (compatibility relation) among the Fourier transforms:
\begin{equation}\label{proj-mu-hat}
\hat \mu_V(v)=\hat \mu_W(i_V^Wv), \qquad V\leq W.
\end{equation}
Let us now define the map $F:\tilde E_A\to \C$ by:
$$F(\eta):=\hat \mu_V(v),\qquad \eta \in \tilde E_A, $$
where  $v\in V$ is any representative of the equivalence class $\eta\in \tilde E_A$. $F$ is unambiguosly defined, indeed given a $v'\sim v$, with $v'\in V'$, there exists a $W\in A$, with $V\leq W$ and $V'\leq W$, such that $i_V^Wv=i_{V'}^Wv'$. By the compatibility condition \eqref{proj-mu-hat} 
$$\hat \mu_V(v)=\hat \mu_W(i_V^Wv)=\hat \mu_W(i_{V'}^Wv')=\hat \mu_{V'}(v').$$
Further, the map $f$ is continuous on $\tilde E_A$ in the final topology.\\
If there exists a measure on $\mu$ on  $E_A$ such that $\mu_V=\pi_V\circ\mu$ for all $V\in A$, then its Fourier transform $\hat\mu$ coincides with $F$ on $\tilde E_A$ and $$\|\hat \mu\|_\infty=\sup_{\eta \in (E_A)*}|\hat \mu| \leq |\mu |,$$
where $|\mu | $ is the total variation of $\mu$. \\
Let us consider the projective system of functionals $(L_V,D(L_V))_{V\in A}$, where $D(L_V)\equiv \Fo(V) $ is the space of functions $f:V\to \C$ of the form
 form $f(v)=\int _Ve^{i\langle v',v\rangle}\nu_f(dv')$ for some complex bounded measure $\nu_f$ on $V$. $D(L_V)$ is a Banach algebra of functions, where the multiplication is the pointwise one and the norm of a function $f\in \Fo(v)$ is the total variation of the associated measure $\mu_f$. Let $L_V:D(L_V)\to \C$ be the linear functional defined by
\begin{eqnarray}
& &D(L_V):=\Fo(V)\nonumber\\
& &L_V(f):=\int_V \hat \mu_V(v)\mu_f(dv)\label{pro-sy-V},
\end{eqnarray}
One can easily verify that $L_V$ is continuous in the $\Fo(V)-$norm.  Indeed 
for any $V\in A$ one has that the map $\hat \mu_V$ is uniformly continuous and bounded, since
$\|\hat \mu_V\|_\infty\leq |\mu_V|$, $|\mu_V|$ denoting the total variation of the measure $\mu_V$. Hence $L_V(f)\leq \|\hat \mu_V(v)\|_\infty|\mu_f|\leq |\mu_V|\|f\|_{\Fo(V)}$. On the other hand by 
Fubini's theorem one has:
\begin{equation}\label{eq-F-P}\int _V f(v)\mu_V(dv)=\int_V \hat \mu_V(v)\mu_f(dv),\end{equation} 
both integrals being absolutely convergent.
The family $(L_V,D(L_V))_{V\in A}$ forms a projective system of functionals. If $\sup_{V\in A}|\mu_V|=+\infty$, according to theorem \ref{th-b-var}, there cannot exist a complex bounded measure $\mu$ on $E_A$ which is the projective limit of the measures $(\mu_V)_{V\in A}$. Hence there cannot exists a projective extension $(L,D(L)$ of the projective system $(L_V,D(L_V))_{V\in A}$ of the form
\begin{eqnarray}
& &D(L):=L^1(E_A,|\mu|)\nonumber\\
& &L(f):=\int_{E_A} f(\omega)\mu(d\omega). \label{pro-ex-mu}
\end{eqnarray}

 However, even if $\mu$ does not exists, the map $F:\tilde E_A\to \C$ is still well defined, and can be used in the construction of a projective extension of $(L_V,D(L_V))_{V\in A}$ alternative to \eqref{pro-ex-mu}.\\
Consider on $\tilde E_A$ the Borel $\sigma-$algebra $\Ba(\tilde E_A)$, then one has that the continuous map $F:\tilde E_A\to \C$ is measurable. 
If the  condition:
\begin{equation}\label{cond-hat-mu}
\sup_{V\in A}\|\hat \mu_V\|_\infty<+\infty
\end{equation}
is satisfied, then the functional $L:D(L)\to \C$ given by 
\begin{eqnarray*}
& &D(L):=\Fo(E_A)\\
& &L(f)=\int _{\tilde E_A} F(\eta)\nu_f(d\eta)
\end{eqnarray*}
is well defined on the Banach algebra $\Fo(E_A)$ of functions $f:E_A\to\C$  of the form $f(\omega)=\int _{\tilde E_A}e^{i\langle \eta,\omega\rangle}\nu_f(d\eta)$ for some complex bounded measure $\nu_f$ on $\tilde E_A$. $L$  is a projective extension of the projective system of functionals \eqref{pro-sy-V}. Indeed considered a cylindrical function $f:E_A\to \C$ of the form $f(\omega) :=g(\pi_V\omega)$, for some $V\in A$ and $g\in \Fo(V)$, $g\equiv \hat \nu_g$,  then $f\in \Fo(E_A)$, as:
\begin{eqnarray*}
f(\omega)&=&g(\pi_V\omega)=\int_Ve^{i\langle v',\pi_V\omega\rangle}\nu_g(dv')=\int_V e^{i\langle i_V v',\omega\rangle}\nu_g(dv')\\
&=&\int_{\tilde E^A} e^{i\langle \eta,\omega\rangle}\nu_f(d\eta),
\end{eqnarray*}
where $\nu_f=i_v\circ \nu_g$.
Indeed:
 \begin{eqnarray*}
L(f)&=&\int _{\tilde E_A} f(\eta)\nu_f(d\eta)=\int _{V} f(i_v v)\nu_g(dv)\\
&=&\int_V\hat \mu_V(v)\nu_g(dv)=L_V(g)
\end{eqnarray*}
 However, for a general $f\in \Fo(E_A)$, contrary to the finite dimensional case, the equality \eqref{eq-F-P} does not make sense and the action of the functional on $f$ cannot be described in terms of an integral with respect to a measure.\\
Depending on the regularity properties of the function $F:\tilde E_A\to \C$ one can define the functional $L$ on different domains. Let $\Ba$ be a Banach space where $\tilde E_A$ is densely embedded, i.e. $\tilde E_A\subset \Ba$,  and let $F$ be continuous with respect to the $\Ba$-norm. Then $F$ can be extended to a function $\tilde F:\Ba\to \C$, with $\tilde F(\eta)=F(\eta)$ for all $\eta \in \tilde E_A$. 
Let $\Fo(\Ba^*)$ be the Banach algebra of functions  $f:\Ba^*\to \C$ of the form 
$$f(x)=\int_\Ba e^{i\langle y,x\rangle}\nu_f(dy),\quad x\in \Ba^*,$$
for some complex bounded variation measure $\nu_f$ on $\Ba$. Then the functional $L':\Fo(\Ba^*)\to \C$ defined by 
\begin{eqnarray*}
& &D(L'):=\Fo(\Ba)\\
& &L'(f)=\int _{\Ba} \tilde F(x)\nu_f(dx)
\end{eqnarray*}
is an alternative projective extension of the system of functionals \eqref{pro-sy-V}. 

A particular example of this approach is provided by the infinite dimensional generalized Fresnel integrals which can be applied to the rigorous mathematical definition of Feynman path integrals \cite{AlHKMa} and, more generally, to the construction of generalized Feynman-Kac type formule representing the solution of PDEs which do not satisfy a maximum principle.\\
Let us illustrate this with a class of examples.

Let $G^p_t$ be the fundamental solution of the PDE:
\begin{equation} \label{PDE-p}\left\{ \begin{array}{l}
\frac{\partial}{\partial t}u(t,x)=(-i)^p \alpha \frac{\partial^p}{\partial x^p}u(t,x)\\
u(0,x) = u_0(x),\qquad x\in\R, t\in [0,+\infty)
\end{array}\right. \end{equation}
where $p\in\N$, $p\geq 2$, and $\alpha\in\C$ is a complex constant. In the following we shall assume that  $|e^{\alpha tx^p}|\leq 1$ for all $x\in\R$ and $ t\in [0,+\infty)$. In particular, if $p$ even  we shall assume that  $\Rea(\alpha)\leq 0$, while if $p$ is odd we shall take  $\alpha$  to be purely imaginary.\\
In the case where $p=2$ and  $\alpha<0$, we obtain the heat equation, while for $p=2$ and $\alpha =-i$ \eqref{PDE-p}  is the Schr\"odinger equation (without potential).
The solution of the Cauchy problem \eqref{PDE-p} is given by:
\begin{equation}\nonumber
u(t,x)=\int_\R G^p_t(x,y)u_0(y)dy,
\end{equation} 
where $G^p_t(x,y)=g^p_t(x-y)$ and $g^p_t\in S'(\R)$ is  the distribution defined as the Fourier transform
\begin{equation}\nonumber
g^p_t(x):=\frac{1}{2\pi}\int e^{ikx}e^{\alpha tk^p}dk,\qquad x\in\R.
\end{equation}
The integral is well defined under restrictions on $u_0$ (e.g. $u_0\in S(\R)$).
For $t=0$  the distribution $g^p_t(x)$ is the Dirac $\delta$ point measure with support at $0$, while if $t>0$ then $g^p_t$ is a smooth function of $x$. In particular if $p=2$ then $g^p_t(x)=\frac{e^{\frac{x^2}{2\alpha t}}}{\sqrt{2\pi (-\alpha) t}}$. If $p>2$, then an explicit form for the function $g_t^p$ cannot be given in terms of elementary functions. On the other hand it is possible to prove that for $p\geq 3$ and   $t>0$  the  tempered distribution $g^p_t\in S'(\R)$ 
belongs to  $C^\infty(\R)\cap L^1(\R)$  (see \cite {Ma14} for a detailed proof).

Let $A$ be the set of finite subsets of the interval $ [0,t]$. For any $J=\{t_1,...,t_n\}\in A$, with $0\leq t_1<...<t_n\leq t$, let us consider the measure $\mu_J$ on $\R^J$ defined by
\begin{multline}\label{mu-J-Gp}
\mu_J(B_1\times...\times B_n)=\\
=\int_{B_1}\dots \int_{B_n}G^p_{t_2-t_1}(x_2,x_1)...G^p_{t-t_{n}}(x_{n+1},x_{n})dx_1...dx_n\delta(x_{n+1}).\end{multline}
In the cases where $p>2$ or when $p=2$ and $\alpha\in \C$ has a non vanishing immaginary part, then one has that the function $g^p_t$ is not real and positive (see \cite{Hoc,Ma14}) and plays the role of the density of a signed measure. Further, as $\int_\R|g^p_t(x)|dx=c>1$, as remarked in section \ref{sez4.2} there cannot exist a $\sigma$-additive (complex or signed) measure $\mu$ on $\R^{[0,t]}$ with finite total variation which is the projective limit of the measures $\mu_J$.
Let us consider instead the projective system of functionals $(L_J, D(L_J))$ defined by
\begin{eqnarray}
 D(L_J)&=&\Fo(\R^J)\nonumber\\
L_J(f)&=&\int _{\R^J}f(x_1,...,x_n)\mu_J(dx_1,...,dx_n), \qquad f\in \Fo(\R^J)\label{f-J-mu}
\end{eqnarray}
where $\Fo(\R^J)$ is the set of functions $f:\R^J\to\C$ of the form $ f(x_1,...,x_n)=\int e^{i\sum_{k=1}^nx_ky_k}\mu_f(dy_1,...,dy_n)$. An alternative projective extension of $(L_J, D(L_J))$ can be constructed as follows.
Fixed a $p\in \N$, with $p\geq 2$, let us consider the Banach space
$\Ba_p$  
\begin{equation}\label{Ba_p}
\Ba_p:=\{\gamma:[0,t]\to\R : \gamma(t)=0,\int_0^t|\dot \gamma(s)|^pds<\infty\}
\end{equation}
of absolutely continuous maps $\gamma:[0,t]\to\R$, with $\gamma(t)=0$ and the weak derivative $\dot \gamma $ belonging to $L^p([0,t])$, endowed with the norm:
$$\|\gamma\|_{\Ba_p}=\left(\int_0^t|\dot \gamma(s)|^pds\right)^{1/p}. $$
One has that the transformation  $T:\Ba_p\to L^p([0,t])$ mapping an element $\gamma\in \Ba_p$ to its weak derivative $\dot \gamma$ is an isomorphism. Its inverse $T^{-1}:L^p([0,t])\to\Ba_p$  is given by:
\begin{equation}\label{T-1}
T^{-1}(v)(s)=-\int_s^t v(u)du\qquad v\in L^p([0,t]), s\in[0,t].
\end{equation}
Analogously the dual space $\Ba_p^*$ is isomorphic to $  L_{q}([0,t])=(L_{p}([0,t]))^*$, with $\frac{1}{p}+\frac{1}{q}=1$, and the pairing between an element $\eta\in\Ba_p^*$ and $\gamma\in \Ba_p$ is given by:
$$\langle \eta, \gamma\rangle=\int_0^t\dot \eta(s)\dot\gamma(s)ds\qquad \dot\eta\in L_{q}([0,t]), \gamma\in\Ba_p.$$
By means of the map \eqref{T-1} it is simple to verify that $\Ba_p^*$ is isomorphic to $\Ba_q$.\\
Let $\Fo(\Ba_q) $ be the space of functions $f:\Ba_q\to\C$ of the form 
$$f(\eta)=\int_{\Ba_p}e^{i\int_0^t\dot \eta(s)\dot \gamma(s)ds}\mu_f(d\gamma),\, \quad \eta \in \Ba_q, \mu_f\in \Mi(\Ba_p).$$
Let $\Phi_p:\Ba_p\to\C$ be the phase function
$$\Phi_p(\gamma):=(-1)^p\alpha\int_0^t\dot\gamma(s)^pds,$$
where  $\alpha\in\C$  is a complex constant such that 
\begin{itemize}
\item $\Rea(\alpha)\leq 0$ if $p$ is even,
\item $\Rea(\alpha)= 0$ if $p$ is odd.
\end{itemize}
In particular, in the case where $p$ is even one has that $\Phi_p(\gamma)$ is proportional to the $\Ba_p$-norm of the vector $\gamma\in \Ba_p$.\\
Let $(L,D(L))$ be the linear functional defined by
\begin{eqnarray}
D(L)&=&\Fo(\Ba_p)\nonumber\\
L(f)&=&\int_{\Ba_p} e^{\Phi_p(x)}\mu_f(dx), \qquad f\in\Fo(\Ba_p).\label{Funz-L-p}
\end{eqnarray}
where $\Fo(\Ba_p)$ is the Banach algebra of functions $f:\Ba_q\to \C$ of the form $f(x)=\int_{\Ba_p}e^{i\langle x,y\rangle}\mu_f(dy)$ for some complex bounded measure $\mu_f$ on $\Ba_p$. The functional $L$ is continuous in the $\Fo(\Ba_p)$-norm since by the assumptions one has $|e^{\Phi_p(x)}|\leq 1$. Hence:
$$|L(f)|\leq \int_{\Ba_p} |e^{\Phi_p(x)}||\mu_f|(dx)\leq \|\mu_f\|=\|f\|_{\Fo (\Ba_p)}$$
gives the continuity of $L$ in the $\|\,\|_\Fo$-norm. Further $L(1)=1$. The functional $L$ is an extension of the projective system of functionals $(L_J,D(L_J))$ defined by \eqref{f-J-mu} as the following result shows.
\begin{theorem}
Let  $f: \Ba_q\to\C$ be a cylindrical function of the 
 following form:
$$f(\eta)=F(\eta(t_1), \eta(t_2), ...,\eta(t_n)),\qquad \eta \in \Ba_q,$$ with $0\leq t_1<t_2<...<t_n< t$ and $F:\R^n\to \C$, $F\in \Fo(\R^n)$: 
$$F(x_1,x_2, ..., x_n)=\int_{\R^n}e^{i\sum_{k=1}^ny_kx_k}\nu_F(dy_1,...,dy_n), \qquad \nu_F\in\Mi(\R^n).$$
Then $f\in\Fo(\Ba_p)=D(L)$ and $L(f)$ is given by:
\begin{equation}\label{IntIphi-p}
L(f)=\int_{\R^n}F(x_1,x_2, ...,x_n)\mu_J(dx_1,dx_2, ...,dx_n),
\end{equation}
where $J=\{t_1,...,t_n\}$ and $\mu_j$ is given by \eqref{mu-J-Gp}.
\end{theorem} 
For a detailed proof see \cite{Ma14}.\\
This particular functional provides in the case $p=2$ and $\alpha=-i$  a mathematical realization of Feynman path integrals and a tool for the construction of the representation of the solution of the Schr\"odinger equation, including also potentials (see, e.g., \cite{AlHK,AlHKMa,AlMa05,Ma} and \cite{HKPS,JoLa,Ma11} for alternative approaches). More generally, in the case where $p>2$, the following result provides a generalization of the Feynman-Kac formula to the case of parabolic equations associated to high-order differential operators (see also \cite{BoMa14,Hoc,Kry,Ma13,Ma14}). Indeed let us consider  a Cauchy problem of the form 
\begin{equation} \label{PDE-p-V}\left\{ \begin{array}{l}
\frac{\partial}{\partial t}u(t,x)=(-i)^p \alpha \frac{\partial^p}{\partial x^p}u(t,x)+V(x)u(t,x)\\
u(0,x) = u_0(x),\qquad x\in\R, t\in [0,+\infty)
\end{array}\right. \end{equation}
where $p\in\N$, $p\geq 2$, and $\alpha\in\C$ is a complex constant such that $|e^{\alpha tx^p}|\leq 1$ forall $x\in\R, t\in [0,+\infty)$, while $V:\R\to\C$ is a bounded continuous function.  We consider for simplicity the case of $x$ on the real line, extensions to $x\in\R^n$ are discussed in \cite{AlHKMa}. Under these assumptions the Cauchy problem \eqref{PDE-p-V} is well posed on $L^2(\R)$. Indeed the operator $\di_p:D(\di_p)\subset L^2(\R)\to L^2(\R)$ defined by 
\begin{eqnarray*}
D(\di_p)&:=& H^p=\{u\in L^2(\R), k\mapsto k^p\hat u(k)\in L^2 \},\\
\widehat{\di_pu}(k)&:=&k^p\hat u(k), \, u\in D(\di_p),
\end{eqnarray*}
($\hat u$ denoting the Fourier transform of $u$) is self-adjoint. For $\alpha\in\C$, with $|e^{\alpha tx^p}|\leq 1$ for all $x\in\R, t\in [0,+\infty)$, one has that  the operator $A:=\alpha D_p$ generates a strongly continuous semigroup $(e^{tA})_{t\geq 0}$ on $L^2(R)$. By denoting by $B:L^2(\R)\to L^2(\R)$ the bounded multiplication operator defined by 
$$Bu(x)=V(x)u(x), \qquad u\in L^2(\R),$$
one  has that the operator sum $A+B:D(A)=D(\di_p)\subset L^2(\R)\to L^2(\R)$ generates a   strongly continuous semigroup $(T(t))_{t\geq 0}$ on $L^2(\R)$. According to the following theorem its action can be described by a generalized Feynman-Kac formula constructed in terms of the linear functional \eqref{Funz-L-p}.
\begin{theorem}\label{Th7}
Let $u_0\in \Fo(\R)\cap L^2(\R)$ and $V\in \Fo(\R)$, with $u_0(x)=\int_\R e^{ixy}\mu_0(dy)$ and $V(x)=\int_\R e^{ixy}\nu(dy)$, $\mu_0,\nu\in\Mi(\R)$.  Then the functional $f_{t,x}:\Ba_q\to\C$ defined by
\begin{equation}\label{f-tx}
f_{t,x}(\eta):=u_0(x+\eta(0))e^{\int_0^tV(x+\eta(s))ds}, \qquad x\in \R, \eta\in \Ba_q,
\end{equation}
belongs to $\Fo(\Ba_q)$ and $L(f_{t,x})$, with $L$ given by \eqref{Funz-L-p}  provides a representation for the solution of the Cauchy problem \eqref{PDE-p-V}, i.e. $u(t,x)=L(f_{t,x})$, where the equality holds for $x$- almost everywhere in $\R$.
\end{theorem}
For a detailed proof see \cite{AlHK} for the case of the Schr\"odinger equation (with $p=2$), and \cite{Ma14} for the case $p>2$.\\

The mathematical theory of Feynman path integrals provides alternative examples of continuous projective limit functionals obtained from projective systems of functions. For instance in \cite{ELT} an alternative construction of a linear functional associated to the representation of the solution of the Schr\"odinger equation, i.e Eq \eqref{PDE-p-V} for $p=2$ and $\alpha=-i/\hbar$ ($\hbar$ being the reduced Planck's constant), has been proposed. The authors construct a functional $(I^\hbar,D(I^\hbar))$ which, in our notation, can be regarded as an extension of the functional \eqref{Funz-L-p} in the case $p=2$.  This particular extension in called {\em infinite dimensional oscillatory integral} because of its relation with classical oscillatory integrals on finite dimensional vector spaces \cite{Hor1,Hor2}(see also \cite{AMBull}). The definition of the functional and  its domain is given in terms of the limit of a sequence of finite dimensional approximations. We recall it for sake of completeness, for further details and for a complete description of this approach and the application to the study of an interesting class of Schr\"odinger-type equations, see \cite{Ma}.

\begin{definition}\label{intoscinf-1}
Let $\Hi=\Ba_2$ be the Hilbert space defined by \eqref{Ba_p}(with $p=2)$ endowed with the inner product
$$(\gamma_1,\gamma_2):=\int_0^t \dot \gamma_1(s)\dot \gamma_2(s)ds.$$
A Borel measurable function $f:\Hi\to\mathbb{C}$ is called $I^\hbar$-integrable 
if  for each  sequence $\{ P_n\}_{n\in\mathbb{N}}$ of projectors
onto n-dimensional subspaces of $\Hi$, such that $P _n\leq P
_{n+1}$ and $P_n \to I$ strongly as $ n \to \infty$ ($I$ being
the identity operator in $\Hi$), the  finite dimensional approximations of the
oscillatory integral of $f$ 
$$ I^\hbar_{P_n}(f):=\int_{P _n\Hi}e ^{\frac{i}{2\hbar}\vert P _nx\vert ^2 }f(P _n x )d(P _nx )\Big( \int _{P _n\Hi}e ^{\frac{i}{2\hbar}\vert P _nx\vert ^2 }d (P _nx )\Big)^{-1},$$
are well defined  (as  improper Lebesgue integrals) and the limit
$\lim  _{ n \to \infty}I^\hbar_{P_n}(f)$
exists and is independent of the sequence $\{ P _n\}$.\\
In this case the limit is  called the infinite dimensional oscillatory integral of
$f$  and is denoted by $$ I^\hbar(f)=\widetilde{\int_\Hi}  e
^{\frac{i}{2\hbar}\vert x\vert ^2 }f(x)dx.$$
\end{definition}

The ``concrete'' description of the  class $D(I^\hbar)$ of all $I^\hbar$-integrable functions  is still an open problem of harmonic analysis, even when $dim (\Hi)<\infty$. The following theorem shows that this class includes $\Fo(\Hi)$, the domain of the functional \eqref{Funz-L-p}.

\begin{theorem}\label{teoCM1} Let $B: \Hi\to\Hi$ be a self-adjoint trace class operator such that $(I-L)$ is invertible ($I$ being the identity operator in $\Hi$). Let us assume that $f\in \Fo(\Hi)$. Then the function $g:\Hi\to\mathbb{C}$ given by 
$$g(x)=e^{-\frac{i}{2\hbar}( x,Bx)}f(x),\qquad x\in \Hi$$ is $I^\hbar$-integrable  and the
corresponding infinite dimensional oscillatory integral $I^\hbar(g)$ is given by the following Cameron-Martin-Parseval 
type formula:
\begin{equation}\label{CMgen} \widetilde{\int{\Hi} }e ^{\frac{i}{2\hbar}(
x,(I-L)x) }f(x)dx=( \det (I-B) )^{-1/2}\int _\Hi e
^{-\frac{i\hbar}{2}( x ,(I-B)^{-1}x ) }\mu_f(dx ) \end{equation} where $\det (I-B)=\vert \det (I-B)\vert
e^{-\pi i \;{\rm   Ind }\; (I-B)}$ is the Fredholm determinant of the
operator $(I-B)$, $\vert \det (I-B)\vert$ its absolute value and Ind($(I-B)$)
is the number of negative eigenvalues of the operator $(I-B)$, counted
with their multiplicity. The functional of $f$ given by \eqref{CMgen} is continuous from $f\in\Fo(\Hi)$ with the norm $\|\mu_f\|$.
\end{theorem}

For a proof see the original paper \cite{ELT}. For extensions of this result with applications to Schr\"odinger equations with potential and magnetic field see \cite{AlBr1,AlBre}.
According to theorem \ref{teoCM1} one can see that the functional $(I^\hbar,D(I^\hbar))$ is a continuous  extension of the functional  \eqref{Funz-L-p}. This enlargement of the class of integrable function and the connection of the functional \eqref{Funz-L-p}  with the solution of the Schr\"odinger equation (on $\R^d$ with potential consisting of a harmonic part plus a part $V\in \Fo(\R^d)$) has been exploited in \cite{AlMa05} for an extension of theorem \ref{Th7} to the case of potentials $V$ with polynomial growth. In the latter case one obtains a representation in terms of expectations od complex-valued functions with respect to the gauusian measure associated to the abstract Wiener space built on the Hilbert space $\Hi$.


\section{Conclusions}

We have introduced the general concept of projective system of complex-valued functionals defined on a subset of the space of complex- valued functions. We discussed projective extensions of such a system to the subsets of the space of complex- valued functions on the projective limit. We proved their existence by constructing a minimal one, whose domain we described in details. We also discussed the uniqueness of projective extensions, as well as continuous extensions. We analysed in particular the ``regular'' case of projective systems of linear functionals associated via integrals with a projective system of complex measure spaces and put in evidence a necessary condition for having a projective limit. The special cases of product measures and projective systems constructed from complex measures were analyzed in details, together with the associated ``pseudo processes''. Concrete examples constructed using fundamental solutions of higher order hyperbolic and parabolic partial differential equations have been discussed and applications to Feynman-Kac- type formulae for such equations with potential were exhibited. We also discussed the situation where the projective extension is not related to an absolutely converging integral, but rather has continuity properties on its own.\\
We showed by examples that our construction is able to cover, next to projective limits given by integrals (with respect to probability or more generally bounded complex measures) also those given by oscillatory integrals of the Feynman- type. This opens up the possibility of a systematic unified study of probabilistic stochastic processes and their analogues described in terms of continuous complex- valued functionals, offering new perspectives to further extensions of the connections between analysis and rigorous path integrals in connection with systems of hyperbolic, in addition to parabolic, (stochastic) partial differential equations. This program will be pursued in further work, including in particular applications to (relativistic) quantum fields and to the study of hyperbolic and parabolic higher order equations on curved spaces.\\
We also intend to continue the study of the manifold of projective extensions, characterizing in particular the maximal ones.


\section*{Acknowledgments}

The second author gratefully acknowledges support by an Alexander von Humboldt Fellowship. The authors are very grateful to the Mathematics Department and CIRM of the University of Trento resp. Institute of Applied Mathematics and HCM of the University of Bonn for support and hospitality.



\end{document}